\documentclass[12pt]{article}
\usepackage{amsmath, amssymb, amsthm, amsfonts}
\usepackage{indentfirst}

\numberwithin{equation}{section}

\theoremstyle{plain}
\newtheorem{theorem}{Theorem}\numberwithin{theorem}{section}
\newtheorem{lemma}{Lemma}\numberwithin{lemma}{section}
\numberwithin{proposition}{section}
\numberwithin{corollary}{section}
\theoremstyle{definition}
\theoremstyle{remark}
\newcommand{\R}{\mathbb{R}}
\newcommand{\s}{\mathbb{S}}

\title {On the Boltzmann equation\\
with the symmetric stable L\'evy process}

\author{Yong-Kum Cho}

\begin{document}

\maketitle

\centerline{Department of Mathematics}
\centerline{College of Natural Science, Chung-Ang University}
\centerline{84 Heukseok-Ro, Dongjak-Gu, Seoul 156-756, Korea}

\medskip

\centerline{e-mail: ykcho@cau.ac.kr}

\bigskip

\begin{itemize}
\item[{}] {\bf Abstract.} As for the spatially homogeneous Boltzmann equation of Maxwellian molecules
with the fractional Fokker-Planck diffusion term, we consider the Cauchy problem for its Fourier-transformed version, which can be
viewed as a kinetic model for the stochastic time-evolution of characteristic functions associated with the symmetric stable L\'evy process
and the Maxwellian collision dynamics. Under a non-cutoff assumption on the kernel, we establish
a global existence theorem with maximum growth estimate, uniqueness and stability of solutions.

\end{itemize}

\bigskip
{\small
\begin{itemize}
\item[{}]{\bf Keywords.} Boltzmann equation, characteristic function, collision, cutoff, equicontinuity,
Fourier transform, L\'evy process, positive definite.

\item[{}] 2010 Mathematics Subject Classification: 35Q82, 47G20, 76P05, 82B40.
\end{itemize}}

\section{Introduction}
We consider the Cauchy problem for the Boltzmann equation, associated to Maxwellian molecules, with an additional diffusion term
in the space-homogeneous setting which reads
\begin{equation}\label{1.1}
\left\{\aligned &{\partial_t f + \delta_p\left(-\Delta\right)^{p/2} f =  Q(f) (v, t) \quad\text{for}\quad (v, t)\in \R^3\times (0, \infty), }\\
& f(v, 0) = f_0(v)
\endaligned\right.
\end{equation}
for the unknown statistical density $\,f= f(v, t)\,$ of particles at velocity $v$ and time $t$, where  $\,0<p\le 2\,$ and
$\,\delta_p\ge 0.\,$ Except for the diffusion or fractional Fokker-Planck term, $\,\delta_p\left(-\Delta\right)^{p/2} f,\,$
it is the Boltzmann equation with the collision term $Q(f)$ defined as
\begin{equation}\label{1.2}
Q(f) (v) = \int_{\mathbb{R}^3}\int_{\mathbb{S}^{2}} b(\mathbf{k}\cdot\sigma)\,
\bigl[f(v') f(v_*') -f(v)f(v_*)\bigr] d\sigma dv_*\,,
\end{equation}
where
\begin{equation}\label{1.3}
\left\{\aligned &{v' = \frac{\,v+v_*\,}{2} +
\frac{\,|v-v_*|\,}{2}\,\sigma\,,}\\
& v_*' = \frac{\,v+v_*\,}{2} - \frac{\,|v-v_*|\,}{2}\,
\sigma\,,\\
&\mathbf{k} \,\,= \frac{v-v_*}{|v-v_*|}
\endaligned\right.
\end{equation}
and $d\sigma$ denotes the area measure on the unit sphere $\,\s^2.\,$

The collision kernel $b$ is an implicitly-defined nonnegative function which represents
a specific type of physical model of collision dynamics
in terms of the deviation angle $\theta$ defined by $\,\cos\theta = \mathbf{k}\cdot\sigma.\,$
For Maxwellian molecules, it is customary to assume that
$b(\cos\theta)$ is supported in $\,[0, \pi/2],\,$ continuous or at least bounded
away from $\,\theta =0\,$ but singular near $\,\theta =0\,$ in such a way that it behaves like
\begin{equation}\label{1.4}
b(\cos\theta)\,\sin\theta \,\sim\, \theta^{-3/2}\quad\text{as}\quad \theta\to 0+
\end{equation}
(see Villani's review paper \cite{V}).

The Cauchy problem (\ref{1.1}) in the case $\,p=2\,$ or the Fokker-Planck-Boltzmann equation
has been studied by a number of authors. For mathematical results as well as relevant physical meanings,
we refer to, with further references therein,
Hamdache \cite{H}, DiPerna \& Lions \cite{DL} in the inhomogeneous setting,
Goudon \cite{Go} in the homogeneous setting and
Bisi, Carrilo \& Toscani \cite{BCT}, Gamba, Panferov \& Villani \cite{GPV} in the inelastic setting.

In this paper, we shall study the Cauchy problem (\ref{1.1}) with $\,0<p\le 2\,$ on the Fourier transform side.
We recall that the Fourier transform of a complex Borel measure $\mu$ on $\R^3$ is defined by
$$\hat{\mu}(\xi) = \int_{\R^3} e^{-i\xi\cdot v}\,d\mu(v)\qquad(\xi\in\R^3),$$
which extends to any tempered distribution $L$ on $\R^3$ as the unique tempered distribution $\widehat{L}$ satisfying
$\,\widehat{L}(\varphi) = L(\hat{\varphi})\,$ for every Schwartz function $\varphi$. If $\mu$ is a
probability measure, that is, a nonnegative Borel measure with unit mass, $\hat{\mu}$ is
said to be a characteristic function.
The fractional Laplacian $(-\Delta)^{p/2}$ is a symbolic notation for the differentiation operator
defined by means of the Fourier transform
\begin{equation}\label{1.5}
\left[\left(-\Delta\right)^{p/2} f\right]\widehat{}\,(\xi) = |\xi|^p\hat{f}(\xi)\qquad(\xi\in\R^3),
\end{equation}
the inverse of the Riesz fractional integration operator of order $p$.

%From a probability-theory point of view, the Cauchy problem (\ref{1.1}), with an initial probability
%density $f_0$, could be considered as a governing equation for
%the time-evolution of a family of probability densities $\{f(\cdot, t)\}_{t\ge 0}$
%and hence it is natural to study the problem on the Fourier transform side
%for it is fundamental in probability theory to investigate a probability distribution
%through its characteristic function.

In \cite{Bo}, Bobylev discovered a remarkable formula
for the Fourier transform of the collision term which reads
\begin{align}\label{1.6}
[Q(f)]\,\widehat{}\,(\xi) &=\int_{\s^2} b\biggl(\frac{\xi\cdot\sigma}{|\xi|}\biggr)
\bigl[\hat f(\xi^+)\hat f(\xi^-) -\hat f(\xi)\hat f (0)\bigr]\,d\sigma\,,\\
&\qquad\xi^+ = \frac{\xi + |\xi|\sigma}{2},\quad \xi^- = \frac{\xi - |\xi|\sigma}{2}\,\nonumber
\end{align}
for each nonzero $\,\xi\in \R^3.\,$ It turns out that Bobylev's formula is legitimate even with singular kernel $b$ for a large
class of functions $f$ (see \cite{Cho} for the derivation in distributional sense).
To simplify our notation, we introduce the {\it Boltzmann-Bobylev} operator $\mathcal{B}$ defined by
\begin{equation}\label{1.7}
\mathcal{B}(\phi)(\xi) = \int_{\s^2} b\biggl(\frac{\xi\cdot\sigma}{|\xi|}\biggr)
\bigl[\phi(\xi^+)\phi(\xi^-) -\phi(\xi)\phi (0)\bigr]d\sigma
\end{equation}
for each complex-valued function $\phi$ on $\R^3$.
In view of Bobylev's formula and definition (\ref{1.5}) of the fractional Laplacian, the Fourier-transformed version of
(\ref{1.1}) takes the form
\begin{equation}\label{1.8}
\left\{\aligned &{\bigl(\partial_t + \delta_p|\xi|^p\bigr)\phi(\xi, t) =  \mathcal{B}(\phi)(\xi, t)\quad \text{for}\quad (\xi, t)\in\R^3\times (0, \infty),}\\
 &\phi(\xi, 0) = \phi_0(\xi)\,.\endaligned\right.
\end{equation}
Considering it as an ordinary differential equation in time and an obvious integrating factor,
it is evident that (\ref{1.8}) is equivalent to
\begin{equation}\label{1.9}
\phi(\xi, t) =
e^{-\delta_p|\xi|^p t}\phi_0(\xi) + \int_0^t  e^{-\delta_p|\xi|^p(t-\tau)}\mathcal{B}(\phi)(\xi, \tau)\,d\tau\,,
\end{equation}
provided that differentiation under the integral sign were permissible.

In the theory of probability, a Markov process $\,\left\{X_t: t\ge 0\right\}\,$ in $\R^3$, with stationary independent increments,
having the family $\,\left\{e^{-|\xi|^p t}: t\ge 0\right\}\,$ as characteristic functions of its continuous transition
probability densities is known as the symmetric stable L\'evy process of index $p$ (see \cite{BG}). Hence
the Cauchy problem (\ref{1.8}), with $\,\delta_p=1,\,$ may be viewed as a kinetic model for the stochastic time-evolution of
characteristic functions governed by the symmetric stable L\'evy process
and the Maxwellian collision dynamics.

In this paper, we are concerned about global existence, uniqueness and stability of solutions
for the Cauchy problem (\ref{1.8}). As for the corresponding Fourier-transformed Boltzmann equation
\begin{equation}\label{1.10}
\left\{\aligned &{\partial_t \phi(\xi, t) =  \mathcal{B}(\phi)(\xi, t)\quad\text{for}\quad (\xi, t)\in \R^3\times (0, \infty), }\\
&\phi(\xi, 0) = \,\phi_0(\xi)\,,
\endaligned\right.
\end{equation}
associated with singular kernel $b$, let us state briefly earlier results.
\begin{itemize}
\item[(i)] In \cite{PT}, Pulvirenti \& Toscani established a global existence of solution to (\ref{1.10}) on the space of
characteristic functions $\phi$ satisfying
\begin{equation}\label{1.111}
\phi(0) =1, \quad \nabla\phi(0) = 0,\quad \Delta\phi(0) = -3
\end{equation}
by using Wild-sum approximation method and also proved uniqueness and stability of solutions in terms of
Tanaka's functional related with probabilistic Wasserstein distance.
\item[(ii)] In \cite{TV}, Toscani \& Villani proved uniqueness and stability,
on the same solution space, with respect to the Fourier-based metric $d_2$
which is a particular case of
\begin{equation}\label{1.11}
d_\alpha(f, g) = \|\phi -\psi\|_\alpha = \sup_{\xi\in\R^3}\,\frac{\,|\phi(\xi) -\psi(\xi)|\,}{|\xi|^\alpha}
\end{equation}
for each $\,\alpha\ge 0\,$ where $\,\phi=\hat f, \,\psi =\hat g\,$ (see also Carrilo \& Toscani \cite{CT} for the properties of
Fourier-based metrics and its various applications).
\item[(iii)] In \cite{CK}, Cannone \& Karch obtained a global existence, uniqueness and stability of solutions
on the space $\mathcal{K}^\alpha$, to be explained below, which turns out to be larger
than the solution space of Pulvirenti \& Toscani.
Quite recently, in \cite{M}, Morimoto completed their work by improving the assumptions on the kernel
and providing another proof of stability. An important feature of a $\mathcal{K}^\alpha$-valued solution $\phi$ of
(\ref{1.10}) is that it may possess infinite energy in the sense $\,\Delta \phi(0, t) = -\infty\,$
for all $\,t\ge 0.\,$ In fact, Bobylev \& Cercignani constructed in \cite{BC} an explicit class of self-similar solutions
having infinite energy, which obviously motivates the work of Cannone \& Karch.
\end{itemize}

In dealing with the Cauchy problem (\ref{1.8}),
we are greatly motivated from the insightful work of Cannone \& Karch and Morimoto.
Let $\mathcal{K}$ denote
the set of all characteristic functions on $\R^3$, that is, complex-valued functions $\phi$
on $\R^3$ identified with
$$\phi(\xi) = \int_{\R^3} e^{-i\xi\cdot v}\,d\mu(v)$$
for some probability measure $\mu$ on $\R^3$. For $\,0<\alpha\le 2,\,$ let
\begin{equation}\label{1.12}
\mathcal{K}^\alpha = \biggl\{\,\phi\in\mathcal{K}\,:\,
\|\phi - 1\|_\alpha = \sup_{\xi\in\R^3}\,\frac{\,|\phi(\xi)-1|\,}{|\xi|^\alpha}<+\infty\,\biggr\}.
\end{equation}
Any characteristic function $\phi$ satisfying (\ref{1.111}) clearly belongs to $\mathcal{K}^\alpha$ for all
$\,0<\alpha\le 2.\,$ As a monotonically indexed family, the embedding
$$\{1\}\subset\mathcal{K}^\beta\subset\mathcal{K}^\alpha\subset\mathcal{K}$$
holds for $\,0<\alpha\le\beta\le 2.\,$ The space $\mathcal{K}^\alpha$ is
a complete metric space with respect to the Fourier-based metric $d_\alpha$ defined in (\ref{1.11})
(for the proofs and further properties, see \cite{CK} and next section).

As it is customary, we denote by $C([0, T]; \mathcal{K}^\alpha)$, with $\,T>0,\,$ the space of all
complex-valued functions $\phi$ on $\,\R^3\times [0, T]\,$ such that
\begin{itemize}
\item[(i)] $\,\phi(\cdot, t)\in\mathcal{K}^\alpha(\R^3)\,$ for each $\,t\in[0, T]\,$ and
\item[(ii)] the map $\,t\mapsto \|\phi(\cdot, t)-1\|_\alpha\,$ is continuous on $[0, T]$.
\end{itemize}
By definition, if $\,\phi\in C([0, T]; \mathcal{K}^\alpha),\,$ then
$\,\phi(\cdot, t)\in C(\R^3)\,$ for each fixed $\,t\ge 0.\,$ In consideration of time-regularity, we shall write
$\,\mathcal{S}^\alpha(\R^3\times [0, T])\,$ for the class of all functions $\,\phi\in C([0, T]; \mathcal{K}^\alpha)\,$
satisfying $\,\phi(\xi, \cdot)\in C([0, T])\,$ and $\,\partial_t\phi(\xi, \cdot)\in C((0, T))\,$ for each fixed $\,\xi\in\R^3,\,$ where
the partial derivative is taken in the usual pointwise sense.  We put
$$\mathcal{S}^\alpha(\R^3\times [0, \infty)) = \bigcup_{T>0}\mathcal{S}^\alpha(\R^3\times [0, T]),$$
which will serve as our solution spaces.

We recall that if the kernel $\,b\in L^1(\s^{2})\,$ in the sense
\begin{align}\label{1.13}
\|b\|_{L^1(\s^{2})} &= \int_{\s^{2}} b(\mathbf{u}\cdot\sigma)\,d\sigma\nonumber\\
&= 2\pi\,\int_0^{\pi/2} b(\cos\theta)\sin\theta\,d\theta <+\infty
\end{align}
for any unit vector $\,\mathbf{u}\in\R^3,\,$ then $b$ is said to satisfy Grad's angular cutoff assumption.
We follow Morimoto to consider weak integrability
\begin{equation}\label{1.14}
\int_0^{\pi/2} b(\cos\theta)\sin\theta \,\sin^{\alpha_0}\left(\frac\theta 2\right)d\theta <+\infty
\end{equation}
with $\,0<\alpha_0\le 2,\,$ which is a kind of quantified non-cutoff conditions on $b$.
It is certainly satisfied by the true Maxwellian kernel $b$ which behaves like
(\ref{1.4}) as long as $\,\alpha_0>1/2.\,$ In addition, we shall consider
\begin{align}\label{1.15}
\lambda_\alpha &= \int_{\s^2} b\biggl(\frac{\xi\cdot\sigma}{|\xi|}\biggr)
\left(\frac{\left|\xi^+\right|^\alpha +  \left|\xi^-\right|^\alpha}{|\xi|^\alpha} -1 \right)\,d\sigma\nonumber\\
&= 2\pi\,\int_0^{\pi/2} b(\cos\theta)\sin\theta\left[\cos^\alpha\left(\frac\theta 2\right) + \sin^\alpha\left(\frac\theta 2\right) -1\right]\,d\theta
\end{align}
for $\,0<\alpha\le 2,\,$ which is independent of $\,\xi\ne 0\,$ and finite under the condition (\ref{1.14}) for all $\,\alpha_0\le\alpha\le 2.\,$
Introduced by Cannone \& Karch, these quantities will serve as the stability exponents.

Our main result for global existence is the following.

\medskip

\begin{theorem}\label{theorem1} Assume that the collision kernel $b$ satisfies a weak integrability condition
(\ref{1.14}) for some $\,0<\alpha_0\le p\,$
and $\,\alpha_0\le\alpha\le p.\,$ Then for any initial datum $\,\phi_0\in\mathcal{K}^\alpha,\,$
there exists a classical solution $\phi$ to the Cauchy problem (\ref{1.8}) in the space
$\,\mathcal{S}^\alpha(\R^3\times [0, \infty))\,$ satisfying
$$ \sup_{\xi\in\R^3}\,e^{\delta_p|\xi|^p t}\bigl|\phi(\xi, t)\bigr|\,\le\, 1\quad\text{for all}\quad t\ge 0.$$
\end{theorem}

A remarkable feature is the maximum growth estimate which says that
the solution stays within the stable L\'evy process of index $p$ for all time.
We should remark that such a behavior of solution is not accidental. In fact, it will be shown
later that any solution satisfies the stated growth estimate under Grad's angular cutoff
assumption on the kernel.

Our main result for stability of solutions and uniqueness is the following theorem for which we define
\begin{align}
\mathcal{S}^\alpha_p(\R^3\times[0,\infty)) = \biggl\{\,&\phi\in \mathcal{S}^\alpha(\R^3\times [0, \infty)):\nonumber\\
&\sup_{\xi\in\R^3}\,e^{\delta_p|\xi|^p t}\bigl|\phi(\xi, t)\bigr|\,\le\, 1\quad\text{for all}\quad t\ge 0\,\biggr\}.
\end{align}

\medskip

\begin{theorem}\label{theorem2}
Assume that the collision kernel $b$ satisfies a weak integrability condition
(\ref{1.14}) for some $\,0<\alpha_0\le p\,$
and $\,\alpha_0\le\alpha\le p.\,$
\begin{itemize}
\item[\rm{(i)}] {\rm{(Stability)}} If $\,\phi, \psi\,$ are solutions to the Cauchy problem (\ref{1.8})
in the space $\,\mathcal{S}^\alpha_p(\R^3\times[0,\infty))\,$
corresponding to the initial data $\,\phi_0, \psi_0\in \mathcal{K}^\alpha,\,$ respectively, then for all $\,t\ge 0\,$
\begin{equation*}
\sup_{\xi\in\R^3} \left[e^{\delta_p |\xi|^p t}\,\frac{\,|\phi(\xi, t) - \psi(\xi, t)|\,}{|\xi|^\alpha}\right]\le e^{\lambda_\alpha t}
\sup_{\xi\in\R^3}\frac{\,|\phi_0(\xi) - \psi_0(\xi)|\,}{|\xi|^\alpha}.
\end{equation*}
\item[\rm{(ii)}] {\rm{(Uniqueness)}} For any initial datum $\,\phi_0\in\mathcal{K}^\alpha,\,$
the Cauchy problem (\ref{1.8}) has at most one solution in the space $\,\mathcal{S}^\alpha_p(\R^3\times[0,\infty)).\,$
\end{itemize}
\end{theorem}

Except for the multiplicative factors of characteristic functions $\,e^{\delta_p |\xi|^p t},\,$
which will play pivotal roles in all of our estimates below, both theorems
are reminiscences of those of Cannone \& Karch and Morimoto.
Concerning the range of $\alpha$, we remark that no local existence results are available in the case $\,\alpha>p\,$ in
the stated solution space for any initial data in the space $\mathcal{K}^\alpha$, which
will be explained in the last section. For the true Maxwellian kernel, we must restrict
$\,1/2<\alpha_0\le\alpha\le p\le 2\,.$

While the detailed computations are quite subtle and different,
our proofs will be done basically along the same patterns as in the work of Cannone \& Karch and Morimoto.
For the existence, we shall prove the theorem
under Grad's angular cutoff assumption first by using the Banach fixed point theory and
then employ a standard limiting argument for the non-cutoff case. For stability of solutions and growth estimate,
our proofs will be carried out along Gronwall-type reasonings and elementary limiting arguments.

As some functionals or expressions involving the space and time variables
are too lengthy to put effectively, we shall often abbreviate the space variables for simplicity in
the sequel.

\section{Characteristic functions}
In this section we collect or establish a list of properties of characteristic functions
which will be basic in the sequel.

\begin{itemize}
\item[(1)] A celebrated theorem of Bochner (\cite{Boc}) states that a complex-valued function $\phi$ on $\R^3$ is the Fourier
transform of a finite nonnegative Borel measure on $\R^3$ if and only if it is continuous and positive definite,
that is, for any integer $N$ and any $\,\xi_1,\cdots, \xi_N\in\R^3,\,z_1,\cdots, z_N\in\mathbb{C},\,$
\begin{equation}\label{C1}
\sum_{i=1}^N\sum_{j=1}^N \phi(\xi_i -\xi_j) z_i\overline{z_j}\,\ge \,0.
\end{equation}
An immediate corollary is that $\phi$ is a characteristic function on $\R^3$ if and only if it is continuous, positive definite and
$\,\phi(0) =1.\,$
\begin{itemize}
\item[(i)] Positive definite functions are closed under finite products, linear combinations with positive coefficients and
pointwise limits. In addition, if $\,\phi\in C(\R^3\times [0, T])\,$ and positive definite in $\xi$ for each $\,t\in[0, T],\,$ then so is
the function
$$\Phi(\xi, t) = \int_0^t\phi(\xi, \tau)\,d\tau\,, $$
where the integral is taken in the Riemann sense.
\item[(ii)] If $\phi$ is positive definite, then the following pointwise estimates are valid for all $\,\xi, \eta\in\R^3\,$
(see \cite{CK}, \cite{LR}):
\begin{align}
\overline{\phi(\xi)} = \phi(-\xi),&\quad\bigl|\phi(\xi)\bigr|\,\le\,\phi(0),\\
\bigl|\phi(\xi) -\phi(\eta)\bigr|^2\,&\le\,2\,{\rm{Re}}\bigl[1 - \phi(\xi-\eta)\bigr],\label{C2}\\
\bigl|\phi(\xi)\phi(\eta) - \phi(\xi +\eta)\bigr|\,&\le\,\left(1-|\phi(\xi)|^2\right)\left(1-|\phi(\eta)|^2\right).\label{C3}
\end{align}
\end{itemize}
\item[(2)] Another fundamental property is known as L\'evy's continuity theorem which asserts that
any pointwise limit of characteristic functions is a characteristic function if it is continuous at the origin (see \cite{LR}).
In particular, if $\,(\varphi_n)\subset\mathcal{K}\,$ and $\,\varphi_n\to\varphi\,$ uniformly on every compact subset of
$\R^3$, then $\,\varphi\in\mathcal{K}.$
\end{itemize}

As for the characteristic functions that arise in stable L\'evy processes, the following non-trivial properties
will be useful in the sequel.

\medskip

\begin{lemma}\label{lemmaW}
For $\,p>0\,,$ let
$$ W_p(\xi, t) = e^{-|\xi|^p t}\quad(\xi\in\R^3, \,t>0)\,.$$
\begin{itemize}
\item[{\rm(1)}] $W_p$ is positive definite in $\xi$ for any $\,t>0\,$ only if $\,0<p\le 2.\,$
\item[{\rm(2)}] For $\,0<p\le 2\,,$ define $\,f_p(v, t) = t^{-3/p} f_p\left(t^{-1/p} v, 1\right)\,$ with
\begin{align*}
f_p(v, 1) = \frac{1}{2\pi^2}\,\int_0^\infty e^{-r^p} r^{2}\left[\frac{\sin(r|v|)}{r|v|}\right]\,dr\,.
\end{align*}
Then $f_p$ is a probability density on $\R^3$ for each $\,t>0\,$ and
$$W_p(\xi, t) =\int_{\R^3} e^{-i\,\xi\cdot v}\,f_p(v, t)\,dv.$$
\begin{itemize}
\item[{\rm(i)}] $f_p(v, t)$ is strictly positive, continuous and
\begin{equation*}
\lim_{|v|\to\infty} |v|^{3+p} f_p(v, t) = \frac{\,p\, 2^{p-1} t\,}{\pi^{5/2}}\,\sin\left(\frac{p\pi}{2}\right) \Gamma\left(\frac{3+p}{2}\right)\Gamma\left(\frac p2\right).
\end{equation*}
\item[{\rm(ii)}] For $\,0\le\alpha<p<2\,,$ $f_p$ satisfies
$$\int_{\R^3} |v|^\alpha f_p(v, t)\,dv <+\infty\quad\text{but}\quad \int_{\R^3} |v|^p f_p(v, t)\,dv =+\infty\,.$$
\end{itemize}
\item[{\rm(3)}] For $\,0<p\le 2\,$ and $\,t>0,\,$ $\,W_p\in\mathcal{K}^\alpha\,$ only if $\,0<\alpha\le p\,$ with
\begin{align*}
\sup_{\xi\in\R^3}\, \frac{\,\left|1- W_p(\xi, t)\right|\,}{|\xi|^p} \,=\, t,\quad
\sup_{\xi\in\R^3}\, \frac{\,\left|1- W_p(\xi, t)\right|\,}{|\xi|^\alpha} \,\le\, t^{\alpha/p}\,.\end{align*}
\end{itemize}
\end{lemma}

\begin{proof} Property (1) is due to Schoenberg (\cite{S}, p. 532).
For property (2), since $\,W_p(\cdot, t)\,$ is radial and integrable with
$$\left\|W_p(\cdot, t)\right\|_{L^1} = \frac{4\pi}{p}\,\Gamma\left(\frac 3p\right)\,,$$
which is easily verified, the representation formula for the probability density $f_p$ is a simple consequence of the Fourier inversion
theorem and the Fourier transform formula for the radial function. The asymptotic identity of (i) is due to
Blumenthal \& Getoor (\cite{BG}, p. 263) and the fractional-moment property (ii) is a direct consequence of (i).

For property (3), observe that
$$\sup_{\xi\in\R^3}\, \frac{\left|1- W_p(\xi, t)\right|}{|\xi|^\alpha} = t^{\alpha/p}\cdot \sup_{r>0}\, \left(\frac{1- e^{-r}}{r^{\alpha/p}}\right)\,.$$
Let $\,E_\alpha(r) = r^{-\alpha/p}(1- e^{-r})\,$ for $\,r>0.\,$ In the case $\,\alpha = p,\,$ it is a smooth decreasing function on $(0, \infty)$
with $\,\lim_{r\to 0+}\,E_p(r) = 1\,$ and hence the first identity follows.
For $\,0<\alpha<p,\,$ it is a smooth function on $\,(0, \infty)\,$ with
$$\,E_\alpha(r)\le \min\,\left(r^{1-\alpha/p}\,,\, 1\right),\quad \lim_{r\to 0+} E_\alpha(r) = \lim_{r\to \infty} E_\alpha(r) =0$$
and hence the second inequality follows. In the case $\,\alpha>p,\,$ we note
$$E_\alpha(r) = r^{1-\alpha/p}\,\int_0^1 e^{-rs}\,ds\,\ge\, r^{1-\alpha/p}\,e^{-r}$$
for which the last term tends to $+\infty$ as $\,r\to 0+.$ Thus
\begin{equation}\label{NP}
\sup_{\xi\in\R^3}\, \frac{\left|1- W_p(\xi, t)\right|}{|\xi|^\alpha} = +\infty.
\end{equation}
\end{proof}

\medskip

\noindent
{\bf Remark 1.}
\begin{itemize}
\item[(i)] In the cases $\,p=1, 2,\,$ the densities are explicitly given by
$$ f_1(v, 1) = \frac{1}{2\pi\,(1+|v|^2)^2}\,,\quad f_2(v, 1) = (4\pi)^{-3/2}\,e^{-|v|^2/4}\,.$$

\item[(ii)] For $\,0<\alpha\le 2,\,$ let $P_\alpha$ denote the set of all probability measures  $\mu$ on $\R^3$
satisfying the finite {\it moment} condition
\begin{equation*}
\int_{\R^3} |v|^\alpha\,d\mu(v) <+\infty
\end{equation*}
together with the additional condition
\begin{equation*}
\int_{\R^3} v_j\,d\mu(v) = 0,\quad j=1,2,3 \quad\text{when}\quad 1<\alpha\le 2.
\end{equation*}
As it is observed by Cannone \& Karch (\cite{CK}, p. 762-763), if $\,\mu\in P_\alpha,\,$ then
$\,\hat{\mu}\in\mathcal{K}^\alpha\,,$ that is, $\,\mathcal{F}(P_\alpha)\subset\mathcal{K}^\alpha\,$
for which $\mathcal{F}$ stands for the Fourier transform operator.  Property (2) shows the reverse implication is false as it exhibits a class of
characteristic function in $\mathcal{K}^\alpha$ whose probability densities do not satisfy the
above finite moment condition, that is, $\,\mathcal{F}(P_\alpha)\subsetneq\mathcal{K}^\alpha\,$ when $\,0<\alpha<2.$
\end{itemize}
\medskip

\begin{lemma}\label{lemmaM}
Let $\,\phi\in\mathcal{K}^\alpha\,$ with $\,0<\alpha\le 2\,.$ Then the following pointwise estimates hold for any $\,\xi, \eta\in\R^3.\,$
\begin{align*}
 &\,\,{\rm (i)}\quad\left|\phi(\xi) - \phi(\eta)\right|^2\,\le\,2\|\phi - 1\|_\alpha\,|\xi -\eta|^\alpha\\
&\,{\rm (ii)}\quad \left|\phi(\xi) - \phi(\eta)\right|\,\le\,\|\phi - 1\|_\alpha\left( 2\sqrt{|\xi|^{\alpha}|\xi -\eta|^{\alpha}} + |\xi -\eta|^\alpha\right)\\
&{\rm (iii)}\quad \left|\phi(\xi) +\phi(\eta) - 2\phi\left(\frac{\xi +\eta}{2}\right)\right| \,\le \,2\|\phi - 1\|_\alpha\,\left|\frac{\xi -\eta}{2}\right|^\alpha
\end{align*}
\end{lemma}

\begin{proof} The first estimate is a simple consequence of (\ref{C2}).
To prove (ii), we first apply (\ref{C3}) and the obvious estimate
$$  1- \left|\phi(\xi)\right|^2 \le \left| 1- \phi(\xi)\right|\left( 1+ |\phi(\xi)|\right) \le 2 \|\phi - 1\|_\alpha|\xi|^\alpha$$
to deduce
$$\left|\phi(\xi)\phi(\eta) - \phi(\xi +\eta)\right|^2 \le 4\|\phi - 1\|_\alpha^2\,|\xi|^\alpha|\eta|^\alpha\,.$$
The desired estimate follows from this one upon observing
\begin{align*}
\left|\phi(\xi) - \phi(\eta)\right| \le \left|\phi(\xi)\phi(\eta -\xi)  - \phi(\eta)\right| + |\phi(\xi)| \left| 1- \phi(\eta -\xi)\right|.
\end{align*}
To prove (iii), write $\,\phi(\xi) = \hat{\mu}(\xi)\,$ with a probability measure $\mu$ on $\R^3$ and
observe the identity
$$\left|  e^{-i\xi\cdot v} +  e^{-i\eta\cdot v} - 2e^{-i\left(\frac{\xi + \eta}{2}\cdot v\right)} \right|
=  2 - e^{-i\left(\frac{\xi - \eta}{2}\cdot v\right)} - e^{i\left(\frac{\xi - \eta}{2}\cdot v\right)} $$
to derive
$$\left|\phi(\xi) +\phi(\eta) - 2\phi\left(\frac{\xi +\eta}{2}\right)\right| \le 2\,{\rm{Re}}\left[1-  \phi\left(\frac{\xi-\eta}{2}\right)\right]$$
from which the desired estimate follows at once.
\end{proof}

\medskip

\noindent
{\bf Remark 2.} The second estimate is due to Morimoto (\cite{M}, p. 555), except a minor adjustment of constant here, which
turns out to play a decisive role in extending the results of Cannone \& Karch.

\section{Fourier-transformed collision terms}
The purpose of this section is to set up a basic computational framework for the Fourier-transformed
collision terms.

In the first place, we consider
\begin{align}\label{F1}
\mathcal{G}(\phi)(\xi) &=  \left\{\aligned &{\int_{\s^2} b\biggl(\frac{\xi\cdot\sigma}{|\xi|}\biggr) \phi(\xi^+)\phi(\xi^-)\,d\sigma}\quad &{\text{for}\quad \xi\ne 0,}\\
&\qquad\qquad\quad\|b\|_{L^1(\s^2)} &\text{for}\quad\xi =0.\endaligned\right.
\end{align}
Under Grad's angular cutoff assumption, the collision operator can be split into two parts $\,Q= Q^+ - Q^-\,$ for which the gain term $Q^+$ is defined as
\begin{equation}\label{F2}
Q^+(f)(v) = \int_{\R^3}\int_{\s^2} b(\mathbf{k}\cdot\sigma) f(v') f(v_*')\,d\sigma dv_*
\end{equation}
and the operator $\mathcal{G}$ arises as its Fourier transform $\,\mathcal{G}(\phi)(\xi) = [Q^+(f)]\,\,\widehat{}\,(\xi)\,$ with $\,\phi(\xi)=\hat f(\xi)\,$ due to
Bobylev's identity.

\medskip

\begin{lemma}\label{lemmaG}
{\rm(cf. \cite{PT})}
Let $\,b\in L^1(\s^2)\,$ and let $\,T>0\,$ be arbitrary.
\begin{itemize}
\item[{\rm(i)}] If $\,\phi\in \mathcal{K},\,$ then $\mathcal{G}(\phi)\,$ is a continuous positive definite function on $\R^3$.
\item[{\rm(ii)}] If  $\,\phi(\cdot, t)\in\mathcal{K}\,$ for each $\,t\in [0, T]\,$ and $\,\phi(\xi, \cdot)\in C([0, T])\,$ for each $\,\xi\in\R^3,\,$
then $\,\mathcal{G}(\phi)\in C(\R^3\times [0, T]).\,$
\end{itemize}
\end{lemma}

\begin{proof}
(i) There exists a unique probability density function $f$ on $\R^3$ so that $\,\phi(\xi)=\hat f(\xi)\,$
and $\,\mathcal{G}(\phi)(\xi) = [Q^+(f)]\,\,\widehat{}\,(\xi)\,$ explained as in the above. Since it is clear $\,Q^+(f)(v)\ge 0\,$ from definition (\ref{F2}) and
$$\int_{\R^3} [Q^+(f)](v) dv = \|b\|_{L^1(\s^2)} \|f\|_{L^1}^2 = \|b\|_{L^1(\s^2)}\,$$
from changing the pre-post collision variables as usual,
we conclude from Bochner's theorem that $\,\mathcal{G}(\phi)\,$ is continuous and positive definite.

(ii) It is enough to prove continuity in time.
Fix $\,\xi\in\R^3,\,t_0\in[0, T]\,$ and consider any $\,(t_k)\subset [0, T]\,$ with $\,t_k\to t_0\,.$
An obvious estimate yields
\begin{align*}
&\bigl| \mathcal{G}(\phi)(\xi, t_k) - \mathcal{G}(\phi)(\xi, t_0)\bigr| \\
&\quad\,\,\le \int_{\s^2} b\biggl(\frac{\xi\cdot\sigma}{|\xi|}\biggr) \bigl[\left|\phi(\xi^+, t_k) - \phi(\xi^+, t_0)\right| +
\left|\phi(\xi^-, t_k) - \phi(\xi^-, t_0)\right| \bigr] d\sigma\,.
\end{align*}
Since the integrands are uniformly bounded by $\,4 b\,,$ we may apply Lebesgue's dominated convergence theorem to conclude
$\,\mathcal{G}(\phi)(\xi, t_k) \to \mathcal{G}(\phi)(\xi, t_0)\,.$
\end{proof}

Next we revisit the Boltzmann-Bobylev operator defined in (\ref{1.7}) on the space of characteristic functions
which becomes
\begin{equation}\label{F3}
\mathcal{B}(\phi)(\xi) =  \int_{\s^2} b\biggl(\frac{\xi\cdot\sigma}{|\xi|}\biggr) \bigl[\phi(\xi^+)\phi(\xi^-) -\phi(\xi)\bigr]\,d\sigma
\end{equation}
for each characteristic function $\phi$ and for $\,\xi\ne 0.\,$ We set $\,\mathcal{B}(\phi)(0) =0.\,$

In order to see if the operator $\mathcal{B}$ makes sense in our function spaces, we shall need to know
a precise way of evaluating the surface integral. Let us fix a non-zero $\,\xi\in\R^3.\,$ By considering a parametrization of
the unit sphere in terms of the deviation angle from the unit vector $\xi/|\xi|,$ we evaluate
\begin{align}\label{F4}
\mathcal{B}(\phi)(\xi) =  \int_0^{\pi/2} b(\cos\theta)\sin\theta
\left\{\int_{\s^1(\xi)}\bigl[\phi(\xi^+)\,\phi(\xi^-)-\phi(\xi)\bigr]\,d\omega\right\}d\theta
\end{align}
in which $\,\s^1(\xi) = \s^2\cap \xi^\perp\,$ and $d\omega$ denotes
the area measure on the unit circle $\,\s^1\subset\R^3.$ As it is defined in (\ref{1.6}), the spherical variables $\,\xi^+, \xi^-\,$ are given by
\begin{align}\label{F6}
\xi^+ &= |\xi| \cos\left(\frac\theta 2\right)\sigma^+,\,\,\, \xi^- = |\xi| \sin\left(\frac\theta 2\right)\sigma^-\quad\text{with}\nonumber\\
&\quad\left\{\aligned &{\sigma^+ = \cos\left(\frac\theta 2\right)\frac{\xi}{|\xi|} + \sin\left(\frac\theta 2\right)\omega,}\\
&\sigma^- = \sin\left(\frac\theta 2\right)\frac{\xi}{|\xi|} - \cos\left(\frac\theta 2\right)\omega.
\endaligned\right.
\end{align}

The following pointwise estimate is extracted from Morimoto (\cite{M}, p. 555), which shows the Boltzmann-Bobylev operator $\mathcal{B}$ is
well-defined on the space $\mathcal{K}^\alpha$ for singular kernel. For the sake of completeness, we reproduce his proof here
in a slightly different setting.

To simplify notation, we shall write
\begin{equation}\label{F30}
\mu_\alpha = 2\pi\,\int_0^{\pi/2} b(\cos\theta)\sin\theta \,\sin^\alpha\left(\frac\theta 2\right)d\theta\,,
\end{equation}
which is finite under the non-cutoff condition (\ref{1.14}) for any $\,\alpha_0\le\alpha\le 2.\,$

\medskip

\begin{lemma}\label{lemmaB}
Let $\,\phi\in\mathcal{K}^\alpha\,$ with $\,0<\alpha\le 2\,$ and $\,\xi\in\R^3-\{0\}.\,$
\begin{itemize}
\item[{\rm(i)}] For each $\,\theta\in (0, \pi/2],$
\begin{equation*}
\left|\int_{\s^1(\xi)} \bigl[\phi(\xi^+)\phi(\xi^-) - \phi(\xi)\bigr]\,d\omega\right|\,\le\,10\pi\,\|\phi - 1\|_\alpha\,|\xi|^\alpha\sin^\alpha\left(\frac\theta 2\right).
\end{equation*}
\item[{\rm(ii)}] If $b$ satisfies the non-cutoff condition $\,\mu_\alpha<+\infty,\,$ then
 \begin{equation*}
\bigl|\mathcal{B}(\phi)(\xi)\bigr|\,\le\,5 \mu_\alpha\,\|\phi - 1\|_\alpha\,|\xi|^\alpha.
\end{equation*}
\end{itemize}
\end{lemma}

\begin{proof} As property (ii) is an immediate consequence of property (i), it suffices to prove property (i). With the representation (\ref{F6}),
we consider
\begin{equation*}
\xi^+_* = |\xi| \cos\left(\frac\theta 2\right){\sigma^+_*}\quad\text{with}\quad
\sigma^+_* = \cos\left(\frac\theta 2\right)\frac{\xi}{|\xi|} - \sin\left(\frac\theta 2\right)\omega
\end{equation*}
for which $\sigma^+_*$ represents the antipodal point of $\sigma^+$ about the parallel circle of the deviation angle $\,\theta/2.\,$
Due to the invariance of the integral under changing variable $\,\omega\mapsto -\omega,\,$ it is evident
\begin{align*}
\int_{\s^1(\xi)} &\bigl[\phi(\xi^+)\phi(\xi^-) - \phi(\xi)\bigr]\,d\omega = {\rm (I) + (II) + (III)},\quad\text{where}\\
{\rm(I)}\,\, &= \frac 12 \int_{\s^1(\xi)} \biggl[\phi(\xi^+) + \phi(\xi^+_*) - 2\phi\left(\frac{\xi^+ + \xi^+_*}{2}\right)\biggr]\,d\omega,\\
{\rm(II)}\, &= \int_{\s^1(\xi)} \biggl[\phi\left(\frac{\xi^+ + \xi^+_*}{2}\right) -\phi(\xi)\biggr]\,d\omega,\\
{\rm(III)} &= \int_{\s^1(\xi)} \phi(\xi^+) \left(\phi(\xi^-) -1 \right)\,d\omega.
\end{align*}
An application of estimate (iii) of Lemma \ref{lemmaM} gives
\begin{align*}
\biggl|\phi(\xi^+) + \phi(\xi^+_*) - 2\phi\left(\frac{\xi^+ + \xi^+_*}{2}\right)\biggr| &\le 2\|\phi -1\|_\alpha \left|\frac{\xi^+ - \xi^+_*}{2}\right|^\alpha\\
&=2\|\phi -1\|_\alpha \left[|\xi|\cos\left(\frac\theta 2\right)\sin\left(\frac\theta 2\right)\right]^\alpha,\end{align*}
which yields immediately the estimate
\begin{equation}\label{A6}
\bigl|{\rm(I)}\bigr|\,\le\, 2\pi \|\phi -1\|_\alpha|\xi|^\alpha\sin^\alpha\left(\frac\theta 2\right).
\end{equation}
As it is straightforward to observe
$$\left|\frac{\xi^+ + \xi^+_*}{2}\right| = |\xi|\cos^2\left(\frac\theta 2\right),\quad
\left|\xi-\frac{\xi^+ + \xi^+_*}{2}\right| = |\xi|\sin^2\left(\frac\theta 2\right),$$
an application of the estimate (ii) of Lemma \ref{lemmaM} gives
\begin{align*}
\biggl|\phi\left(\frac{\xi^+ + \xi^+_*}{2}\right)-\phi(\xi)\biggr|\le \|\phi-1\|_\alpha|\xi|^\alpha \sin^\alpha\left(\frac\theta 2\right)\left[2\cos^\alpha\left(\frac\theta 2\right) +
\sin^\alpha\left(\frac\theta 2\right)\right],
\end{align*}
which yields the estimate
\begin{equation}\label{A7}
\bigl|{\rm(II)}\bigr|\,\le\, 6\pi \|\phi -1\|_\alpha|\xi|^\alpha\sin^\alpha\left(\frac\theta 2\right).
\end{equation}
Since it is trivial to see
\begin{equation}\label{A8}
\bigl|{\rm(III)}\bigr|\,\le\, 2\pi \|\phi -1\|_\alpha|\xi|^\alpha\sin^\alpha\left(\frac\theta 2\right),
\end{equation}
the desired estimate follows upon adding (\ref{A6}), (\ref{A7}) and (\ref{A8}).
\end{proof}

As an application, we have the following time-continuity property of the Boltzmann-Bobylev
operator which will be basic throughout the rest.

\medskip

\begin{lemma}\label{lemmaY} For $\,0<\alpha\le 2,\,$ assume that $b$ satisfies the non-cutoff condition
$\,\mu_\alpha<+\infty\,$ and $\,T>0.\,$ If $\,\phi\in C([0, T]; \mathcal{K}^\alpha)\,$ and $\,\phi(\xi, \cdot)\in C([0, T])\,$ for each
$\,\xi\in\R^3,\,$ then $\,\mathcal{B}(\phi)(\xi, \cdot)\in C([0, T])\,$ for each $\,\xi\in\R^3.\,$
\end{lemma}

\begin{proof}
Fix a non-zero $\,\xi\in\R^3\,$ and $\,t_0\in [0, T].\,$ For any sequence $\,(t_n)\subset [0, T]\,$ with $\,t_n\to t_0,\,$
we may write
\begin{align*}
\mathcal{B}(\phi)(\xi, t_n) &=  \int_{\s^2} b\biggl(\frac{\xi\cdot\sigma}{|\xi|}\biggr) \bigl[\phi(\xi^+, t_n)\phi(\xi^-, t_n) -\phi(\xi, t_n)\bigr]\,d\sigma\\
&=\int_0^{\pi/2} b(\cos\theta)\sin\theta\,I_n(\xi, \theta)\,d\theta\quad\text{where}\\
I_n(\xi, \theta) &= \int_{\s^1(\xi)}\bigl[\phi(\xi^+, t_n)\phi(\xi^-, t_n)-\phi(\xi, t_n)\bigr]\,d\omega\,.
\end{align*}
By the first estimate of Lemma \ref{lemmaB}, we notice
\begin{align*}
\bigl|I_n(\xi, \theta)\bigl| &\le 10\pi\,\|\phi(t_n) - 1\|_\alpha\,|\xi|^\alpha\sin^\alpha\left(\frac\theta 2\right)\\
&\le 10\pi\, C_\alpha(T)\,|\xi|^\alpha \sin^\alpha\left(\frac\theta 2\right)
\end{align*}
where we put $\,C_\alpha(T) = \max_{t\in[0, T]}\|\phi(t) -1\|_\alpha.\,$ Since
$\,\phi\in C([0, T]; \mathcal{K}^\alpha),\,$ the continuity of $\,t\mapsto \|\phi(t)-1\|_\alpha\,$ implies $\,C_\alpha(T)<+\infty.\,$
We now observe that the integrands $\,b(\cos\theta)\sin\theta\,I_n(\xi, \theta)\,$ are uniformly dominated by
$$ I(\xi, \theta) = 10\pi\, C_\alpha(T)\,|\xi|^\alpha \,b(\cos\theta)\sin\theta\sin^\alpha\left(\frac\theta 2\right).$$
Since
$$\int_0^{\pi/2} I(\xi, \theta)\,d\theta  = 5\mu_\alpha\,C_\alpha(T)\,|\xi|^\alpha <+\infty,$$
we may apply Lebesgue's dominated convergence theorem to conclude
$$\lim_{n\to\infty} \mathcal{B}(\phi)(\xi, t_n) = \mathcal{B}(\phi)(\xi, t_0),$$
which proves continuity at $t_0$.
\end{proof}

\section{Results in the cutoff case}
In this section we shall construct a solution to the Cauchy problem (\ref{1.8})
under Grad's angular cutoff assumption on the kernel
and prove uniqueness and stability of solutions. Our principal result is the following.

\medskip

\begin{theorem}\label{theoremC}
Assume that $\,b\in L^1(\s^2)\,$ and $\,\phi_0\in\mathcal{K}^\alpha\,$ for $\,0<\alpha\le p.\,$
\begin{itemize}
\item[{\rm(i)}] There
exists a classical solution to the Cauchy problem (\ref{1.8}) in the space $\mathcal{S}^\alpha(\R^3\times [0, \infty))$.
\item[{\rm(ii)}] If $\phi$ is any classical solution
to the Cauchy problem (\ref{1.8}) in the space $\mathcal{S}^\alpha(\R^3\times [0, \infty)),$ then
necessarily it satisfies
$$ \sup_{\xi\in\R^3}\,e^{\delta_p|\xi|^p t}\bigl|\phi(\xi, t)\bigr|\,\le\, 1\quad\text{for all}\quad t\ge 0.$$
\end{itemize}
\end{theorem}

We shall prove the theorem in several steps.
For the sake of convenience, we introduce the quantities
\begin{align}\label{g}
\gamma_\alpha &= \int_{\s^2} b\biggl(\frac{\xi\cdot\sigma}{|\xi|}\biggr)
\left(\frac{\left|\xi^+\right|^\alpha +  \left|\xi^-\right|^\alpha}{|\xi|^\alpha}\right)\,d\sigma\nonumber\\
&= 2\pi\,\int_0^{\pi/2} b(\cos\theta)\sin\theta\left[\cos^\alpha\left(\frac\theta 2\right) + \sin^\alpha\left(\frac\theta 2\right)\right]\,d\theta
\end{align}
for $\,0<\alpha\le 2,\,$ which will appear frequently in the sequel. Of course, Grad's angular cutoff
assumption is simply $\,\gamma_2 = \|b\|_{L^1(\s^2)}<+\infty\,$ and it is trivial to check
$\,\gamma_2\le\gamma_\alpha\le 2\gamma_2\,$ for all $\,0<\alpha\le 2.\,$

\subsection{Global existence}
For global existence part, we shall construct a solution $\,\phi\in C([0, \infty); \mathcal{K}^\alpha)\,$
to the alternative integral equation
\begin{equation}\label{E1}
\phi(\xi, t) =
e^{-(\gamma_2 + \delta_p|\xi|^p)t}\phi_0(\xi) + \int_0^t  e^{-(\gamma_2 + \delta_p|\xi|^p)(t-\tau)}\mathcal{G}(\phi)(\xi, \tau)\,d\tau
\end{equation}
with the additional properties that $\phi$ and the integrand on the right side are continuous in time. By the fundamental theorem of
calculus, it is plain to find that such a solution $\phi$ belongs to the stated solution space and satisfies the Cauchy problem (\ref{1.10})
in the strict classical sense (here and below time-integration is always taken in the definite Riemann integral sense).

Let $\,T>0\,$ be arbitrary and consider
\begin{equation}\label{E2}
\Omega_T = \left\{\phi\in C\left([0, T]; \mathcal{K}^\alpha\right): \phi(\xi, \cdot)\in  C([0, T])\quad\text{for each}\quad \xi\in\R^3\right\}.
\end{equation}
Since $\mathcal{K}^\alpha$ is a complete metric space, $\Omega_T$ is also a complete metric space with respect to the induced metric
\begin{equation}\label{E3}
\rho_T(\phi, \psi) = \max_{t\in[0, \,T]}\,\|\phi(t) -\psi(t)\|_\alpha\,.
\end{equation}

We set
\begin{equation}\label{E4}
\mathcal{A}(\phi)(\xi, t) =
e^{-(\gamma_2 + \delta_p|\xi|^p)t}\phi_0(\xi) + \int_0^t  e^{-(\gamma_2 + \delta_p|\xi|^p)(t-\tau)}\mathcal{G}(\phi)(\xi, \tau)\,d\tau.
\end{equation}
By Lemma \ref{lemmaM}, if $\,\phi\in\Omega_T,\,$ then $\,\mathcal{G}(\phi)(\xi, \cdot)\in C([0, T])\,$ so that the integral
is well-defined on $\Omega_T$. As it is standard, we shall prove local existence
via the Banach contraction mapping principle applied to the operator $\mathcal{A}$ on $\Omega_T$ with an appropriate choice of $T$ and global existence
via a bootstrap argument.

As the first step, we begin with

\medskip

\begin{lemma}\label{lemmaN}
For $\,0<\alpha\le p\,$ and any $\,T>0,\,$ if $\,\phi\in\Omega_T,\,$ then
\begin{align}
&\left\|\mathcal{A}(\phi)(t) - 1 \right\|_\alpha\,\le\, (\delta_p t)^{\alpha/p} + 2\left(\max_{\tau\in[0, \,T]}\|\phi(\tau)-1\|_\alpha\right),\label{NI}\\
&\qquad\left\|\mathcal{A}(\phi)(t) - \mathcal{A}(\phi)(s)\right\|_\alpha\,\le\, C(\phi, T)\,|t-s|^{\alpha/p}\label{NII}
\end{align}
for all $\,s, t\in [0, T],\,$ where $\,C(\phi, T)\,$ is a constant given below in (\ref{NC4}).
\end{lemma}

\begin{proof}
In view of $\,\mathcal{G}(\phi)(0, t) = \gamma_2\,$ for all $\,t\ge 0,\,$ it is routine to check
\begin{align*}
\mathcal{A}(\phi)(t) &-1  = I(t) + J(t) + K(t) \quad\text{where}\\
I(t) &=  e^{-\gamma_2 t}\,\left[e^{-\delta_p|\xi|^p t} \phi_0(\xi) -1\right]\,,\\
J(t) &= \int_0^t  e^{-(\gamma_2 + \delta_p|\xi|^p)(t-\tau)}\,\left[\mathcal{G}(\phi)(\xi, \tau) - \mathcal{G}(\phi)(0, \tau)\right]\,d\tau\,,\\
K(t) &= \gamma_2\,\int_0^t  e^{-\gamma_2 (t-\tau)}\,\left[ e^{-\delta_p|\xi|^p (t-\tau)} - 1 \right]\,d\tau\,.
\end{align*}
\begin{itemize}
\item[(i)] Writing
$$e^{-\delta_p|\xi|^p t} \phi_0(\xi) -1 = e^{-\delta_p|\xi|^p t}(\phi_0(\xi) -1) - \left( 1- e^{-\delta_p|\xi|^p t}\right)$$
and making use of (3) of Lemma \ref{lemmaW}, it is easy to see
\begin{equation}\label{NI1}
\sup_\xi\,\frac{|I(\xi, t)|}{|\xi|^\alpha}\,\le\,  e^{-\gamma_2 t}\, \left[\|\phi_0 -1 \|_\alpha + (\delta_p t)^{\frac\alpha p}\right]\,.
\end{equation}
\item[(ii)] By estimating in an obvious way
$$\left|\phi(\xi^+, \tau)\phi(\xi^-, \tau) - 1\right| \le \|\phi(\tau) -1\|_\alpha \,(|\xi^+|^\alpha + |\xi^-|^\alpha),$$
it is plain to deduce
\begin{align}\label{NI2}
\sup_\xi\,\frac{|J(\xi, t)|}{|\xi|^\alpha}\,&\le\,\gamma_\alpha\int_0^t  e^{-\gamma_2 (t-\tau)}\,\|\phi(\tau) -1\|_\alpha\,d\tau\nonumber\\
&\le\,\gamma_\alpha\left(\frac{1-e^{-\gamma_2 t}}{\gamma_2}\right)\,\left(\max_{\tau\in[0,\, T]}\|\phi(\tau) -1\|_\alpha\right).
\end{align}
\item[(iii)] Making use of (3) of Lemma \ref{lemmaW} again, we estimate
\begin{align}\label{NI3}
\sup_\xi\,\frac{|K(\xi, t)|}{|\xi|^\alpha}\,&\le\,\gamma_2\int_0^t  e^{-\gamma_2 (t-\tau)}\,\left[\delta_p(t-\tau)\right]^{\frac\alpha p}\,d\tau\nonumber\\
&\le\,\left(\delta_p t\right)^{\frac\alpha p}\left(1-e^{-\gamma_2 t}\right)\,.
\end{align}
\end{itemize}

Adding (\ref{NI1})-(\ref{NI3}) with noting $\,\gamma_\alpha\le 2\gamma_2,\,$ we obtain (\ref{NI}).

To prove the second estimate, fix $\,s, t\in [0, T]\,$ with $\,s<t\,$ for simplicity.
Keeping the same notation as above, we can write
$$\mathcal{A}(\phi)(t) -\mathcal{A}(\phi)(s) = [I(t)-I(s)] + [J(t)-J(s)] + [K(t) - K(s)].$$

\begin{itemize}
\item[(i)] Upon writing
\begin{align*}
\frac{I(t) - I(s)}{|\xi|^\alpha} &= - e^{-\gamma_2 s}\,\biggl\{ e^{-\delta_p|\xi|^p s}\left(\frac{1-  e^{-\delta_p|\xi|^p (t-s)}}{|\xi|^\alpha}\right)\,\phi_0(\xi)\\
&\quad + \left( 1- e^{-\gamma_2 (t-s)}\right)\left(\frac{e^{-\delta_p|\xi|^p t} \phi_0(\xi) -1}{|\xi|^\alpha}\right)\biggr\},
\end{align*}
we proceed as before to deduce
\begin{align}\label{NC1}
&\sup_\xi\,\frac{|I(\xi, t)- I(\xi, s)|}{|\xi|^\alpha}\nonumber\\
&\qquad\le\,  e^{-\gamma_2 s}\, \biggl\{ [\delta_p (t-s)]^{\frac\alpha p}  + \gamma_2 (t-s)
\left[\|\phi_0 -1 \|_\alpha + (\delta_p t)^{\frac\alpha p}\right]\biggr\}\nonumber\\
&\qquad \le\,|t-s|^{\frac\alpha p}\biggl\{ \delta_p^{\frac\alpha p}  + \gamma_2 T^{1-\frac\alpha p}
\left[\|\phi_0 -1 \|_\alpha + (\delta_p T)^{\frac\alpha p}\right]\biggr\}.
\end{align}

\item[(ii)] We write
\begin{align*}
&\frac{J(t) - J(s)}{|\xi|^\alpha} =\int_s^t  e^{-(\gamma_2 + \delta_p|\xi|^p)(t-\tau)}\biggl[\frac{\mathcal{G}(\phi)(\xi, \tau) - \mathcal{G}(\phi)(0, \tau)}{|\xi|^\alpha}\biggr] d\tau\\
&\,\,-\left( 1- e^{-\gamma_2(t-s)}\right)\int_0^s  e^{-(\gamma_2 + \delta_p|\xi|^p)(s-\tau)}\biggl[\frac{\mathcal{G}(\phi)(\xi, \tau) - \mathcal{G}(\phi)(0, \tau)}{|\xi|^\alpha}\biggr] d\tau\\
&\,\, - \biggl[\frac{1- e^{-\delta_p|\xi|^p(t-s)}}{|\xi|^\alpha}\biggr]\int_0^s  e^{-\gamma_2 (t-\tau)-\delta_p|\xi|^p(s-\tau)}\left[\mathcal{G}(\phi)(\xi, \tau) - \mathcal{G}(\phi)(0, \tau)\right] d\tau.
\end{align*}
By using the estimate
$$ \left|\mathcal{G}(\phi)(\xi, \tau) - \mathcal{G}(\phi)(0, \tau)\right|\,\le\,\min \bigl\{ 2\gamma_2,\, \gamma_\alpha \|\phi(\tau) -1\|_\alpha|\xi|^\alpha\bigr\},$$
readily verified as in (\ref{NI2}), and (3) of Lemma \ref{lemmaW}, we deduce
\begin{align}\label{NC2}
&\sup_\xi\,\frac{|J(\xi, t)- J(\xi, s)|}{|\xi|^\alpha}\,\le\,
\gamma_\alpha\left(\max_{\tau\in[0,\, T]}\|\phi(\tau) -1\|_\alpha\right)\nonumber\\
&\quad\qquad\quad\times\quad \left\{\int_s^t e^{-\gamma_2(t-\tau)} d\tau + (t-s)\gamma_2\int_0^s  e^{-\gamma_2(s-\tau)} d\tau\right\}\nonumber\\
&\quad\qquad\quad +\quad 2[\delta_p(t-s)]^{\frac\alpha p}\gamma_2\int_0^s  e^{-\gamma_2(t-\tau)} d\tau\nonumber\\
&\qquad \le\, |t-s|^{\frac\alpha p}\biggl\{ 2\gamma_\alpha\,T^{1-\frac\alpha p}\,\left(\max_{\tau\in[0, \,T]}\|\phi(\tau) -1\|_\alpha\right)
+ 2\delta^{\frac\alpha p}\biggr\}.
\end{align}

\item[(iii)] We write
\begin{align*}
&\frac{K(t) - K(s)}{|\xi|^\alpha} = -\gamma_2\int_s^t  e^{-\gamma_2 (t-\tau)}\biggl[\frac{ 1- e^{-\delta_p|\xi|^p(t-\tau)}}{|\xi|^\alpha}\biggr] d\tau\\
&\qquad -\gamma_2\biggl[\frac{1- e^{-\delta_p|\xi|^p(t-s)}}{|\xi|^\alpha}\biggr]\int_0^s e^{-\gamma_2 (t-\tau)-\delta_p|\xi|^p(s-\tau)} d\tau\\
&\qquad +\gamma_2 \left( 1- e^{-\gamma_2(t-s)}\right)\int_0^s e^{-\gamma_2(s-\tau)}\biggl[\frac{1- e^{-\delta_p|\xi|^p(s-\tau)}}{|\xi|^\alpha}\biggr] d\tau
\end{align*}
and proceed to estimate
\begin{align}\label{NC3}
&\sup_\xi\,\frac{|K(\xi, t)- K(\xi, s)|}{|\xi|^\alpha}\,\le\,\gamma_2 \int_s^t  e^{-\gamma_2 (t-\tau)}[\delta_p(t-\tau)]^{\frac\alpha p}d\tau\nonumber\\
&\qquad + \gamma_2[\delta_p(t-s)]^{\frac\alpha p}\,\int_0^s e^{-\gamma_2 (s-\tau)} d\tau\nonumber\\
&\qquad + \gamma_2^2 (t-s) \int_0^s  e^{-\gamma_2 (s-\tau)}[\delta_p(s-\tau)]^{\frac\alpha p}d\tau\nonumber\\
% &\qquad\quad \le 2[\delta_p(t-s)]^{\frac\alpha p} + \gamma_2(t-s)[\delta_p T]^{\frac\alpha p}\nonumber\\
&\qquad \le (\delta_p|t-s|)^{\frac \alpha p}(2+\gamma_2 T)
\end{align}
\end{itemize}

Upon adding (\ref{NC1})-(\ref{NC3}), we obtain the claimed estimate (\ref{NII}) with
\begin{equation}\label{NC4}
C(\phi, T) = \delta_p^{\frac\alpha p} (5+ 2\gamma_2 T) + 3\gamma_\alpha T^{1-\frac\alpha p}\,\left(\max_{\tau\in[0, T]}\|\phi(\tau) -1\|_\alpha\right)\,.
\end{equation}
\end{proof}

\bigskip

\paragraph{Proof of existence.}
We first claim that $\mathcal{A}$ maps $\Omega_T$ into itself for any arbitrary $\,T>0.\,$
Indeed, the claim follows upon combining the following itemized properties valid for each fixed $\,\phi\in\Omega_T.\,$

\begin{itemize}
\item[(i)] $\,\mathcal{A}(\phi)(\cdot, t)\in\mathcal{K}\,$ for all $\,t\in[0, T]\,$:

By Lemma \ref{lemmaG} and the general properties of positive definite functions stated in the beginning part of section 2,
$\mathcal{A}(\phi)$ is a continuous positive definite function in $\,\xi\in\R^3\,$. Since $\,\mathcal{A}(\phi)(0, t) =1,\,$
the claim follows.

\item[(ii)] $\,\mathcal{A}(\phi)(\cdot, t)\in\mathcal{K}^\alpha\,$ for all $\,t\in[0, T]\,$:

This is an immediate consequence of (i) and (\ref{NI}) of Lemma \ref{lemmaN}.

\item[(iii)] The map $\,t\mapsto \left\|\mathcal{A}(\phi)(t)-1\right\|_\alpha\,$ is continuous on $\,[0, T]\,$:

This is an immediate consequence of (\ref{NII}) of Lemma \ref{lemmaN} for
$$\biggl|\bigl\|\mathcal{A}(\phi)(t) - 1\bigr\|_\alpha - \bigl\|\mathcal{A}(\phi)(s) - 1\bigr\|_\alpha\biggr|
\le \left\|\mathcal{A}(\phi)(t) - \mathcal{A}(\phi)(s)\right\|_\alpha\,.$$

\item[(iv)] $\,\mathcal{A}(\phi)(\xi, \cdot)\in C([0, T])\,$ for all $\,\xi\in\R^3\,$:

By Lemma \ref{lemmaG} again, this claim follows trivially.
\end{itemize}

We now choose $T_0$ satisfying $\,0<T_0<\ln 2/\gamma_2\,$ and claim that $\mathcal{A}$ is a contraction
on the space $\Omega_{T_0}$. To see this, let us take any pair $\,\phi, \psi\in\Omega_T.\,$ By making use of the evident estimate
$$\left|\phi(\xi^+, \tau)\phi(\xi^-, \tau) -\psi(\xi^+, \tau)\psi(\xi^-, \tau)\right|
\le \|\phi(\tau) -\psi(\tau)\|_\alpha \,(|\xi^+|^\alpha + |\xi^-|^\alpha)$$
valid for all $\,\xi\in\R^3, \,\sigma\in\s^2,\,\tau\in [0, t],\,$ we easily find
$$\left\|\,\mathcal{G}(\phi)(\tau) - \mathcal{G}(\psi)(\tau)\,\right\|_\alpha \,\le\,
\gamma_\alpha\,\|\phi(\tau) -\psi(\tau)\|_\alpha\,.$$
Consequently, for any $\,\xi\in\R^3\,$ and $\,t\in[0, T_0],$ we deduce
$$\frac{\,\left|\mathcal{A}(\phi)(\xi, t) - \mathcal{A}(\psi)(\xi, t)\,\right|}{|\xi|^\alpha}$$
is bounded by, since $\,\gamma_\alpha\le2\gamma_2,\,$
\begin{align*}
\int_0^t  e^{-(\gamma_2 + \delta_p|\xi|^p)(t-\tau)} &\frac{\left|\mathcal{G}(\phi)(\xi, \tau) - \mathcal{G}(\psi)(\xi, \tau)\right|}{|\xi|^\alpha}\,d\tau\\
&\le \gamma_\alpha \int_0^t  e^{-\gamma_2(t-\tau)}\|\phi(\tau) -\psi(\tau)\|_\alpha\,d\tau\\
&\le 2\left(1- e^{-\gamma_2 t}\right)\max_{\tau\in [0, \,T_0]}\|\phi(\tau) -\psi(\tau)\|_\alpha\,.
%&\le 2\left(1- e^{-\gamma_2 T_0}\right)\,d_{T_0}(\phi, \psi)\,.
\end{align*}
Taking supremum over $\,\xi\in\R^3,\,$ we obtain
\begin{equation}
\rho_{T_0}\left(\mathcal{A}(\phi),\,\mathcal{A}(\psi)\right)\,\le\,  2\left(1- e^{-\gamma_2 T_0}\right)\,\rho_{T_0}(\phi, \psi)
\end{equation}
and the claim follows for $\, 2\left(1- e^{-\gamma_2 T_0}\right)<1.\,$

By the Banach contraction mapping principle, hence, the operator $\mathcal{A}$ has a unique fixed point on the space
$\Omega_{T_0},$ which is a unique solution to the integral equation (\ref{E1}) on $[0, T_0].$
Since the choice of $T_0$ is independent of the initial datum and depends only on $\gamma_2$, we may repeat the
same argument on $[T_0, 2T_0]$, with $\phi(\xi, T_0)$ as the new initial datum and an obvious
modification of $\Omega_{T_0}$, to construct a unique solution on $[T_0, 2T_0]$. By gluing the two solutions together,
we obtain a solution on $[0, 2T_0]$. By repeating this procedure, we obtain a solution to (\ref{E1})
on any finite time interval. This completes the proof for the existence part in Theorem \ref{theoremC}.

\subsection{Maximum growth estimate}
We next prove that any solution $\phi$ constructed as above satisfies the stated growth estimate.
In fact, the following {\it a priori} estimate holds for any solution to (\ref{E1}) in a less restrictive setting.

\medskip

\begin{lemma}\label{lemmaMa} Let $\,b\in L^1(\s^2)\,$ and $\,\phi_0\in \mathcal{K}.\,$
If $\phi$ is a solution to (\ref{E1}) such that $\,\phi\in C([0, \infty); \mathcal{K})\,$ and $\,\phi(\xi, \cdot)\in C([0, \infty))\,$
for each $\,\xi\in \R^3,\,$ then
\begin{equation*}
\sup_{\xi\in\R^3}\,e^{\delta_p|\xi|^p t}\bigl|\phi(\xi, t)\bigr|\,\le\, 1\quad\text{for all}\quad t\ge 0.
\end{equation*}
\end{lemma}

\begin{proof}
If we set $\,\Phi(\xi, t) = e^{(\gamma_2 + \delta_p|\xi|^p)t}\,\phi(\xi, t),\,$ then it satisfies
$$\Phi(\xi, t) = \phi_0(\xi) + \int_0^t  e^{(\gamma_2 + \delta_p|\xi|^p)\tau}\,\mathcal{G}(\phi)(\xi, \tau)\,d\tau\,.$$
By using the elementary inequality $\,|\xi|^p \le \left|\xi^+\right|^p + \left|\xi^-\right|^p\,$
for $\,0<p\le 2\,$ and for each $\,\sigma\in\s^2\,,$ we find
\begin{align*}
& e^{\delta_p\tau |\xi|^p }\left|\mathcal{G}(\phi)(\xi, \tau)\right| \\
&\qquad\quad\le \int_{\s^2}
b\left(\frac{\xi\cdot\sigma}{|\xi|}\right) \bigl[ e^{\delta_p\tau |\xi^+|^p}\left|\phi(\xi^+, \tau)\right|\bigr]
\bigl[ e^{\delta_p\tau |\xi^-|^p}\left|\phi(\xi^-, \tau)\right|\bigr] d\sigma\\
&\qquad\quad\le e^{-2\gamma_2\,\tau}\,\int_{\s^2}b\left(\frac{\xi\cdot\sigma}{|\xi|}\right)
\left|\Phi(\xi^+, \tau)\right|\,\left|\Phi(\xi^-, \tau)\right|\,d\sigma\\
&\qquad\quad\le \gamma_2\,e^{-2\gamma_2\,\tau}\,\left\|\Phi(\cdot, \tau)\right\|_\infty^2\,,
\end{align*}
where we put $\,\left\|\Phi(\cdot, \tau)\right\|_\infty = \sup_{\xi\in\R^3}\, \left|\Phi(\xi, \tau)\right|.\,$ Consequently,
\begin{equation}\label{G1}
\left\|\Phi(\cdot, t)\right\|_\infty \le 1 + \gamma_2\int_0^t e^{-\gamma_2 \tau}\left\|\Phi(\cdot, \tau)\right\|_\infty^2\,d\tau\,.
\end{equation}
If we denote by $U(t)$ the right side of (\ref{G1}), then it satisfies
$$ \frac{d}{dt}\, U(t) \le \gamma_2 \,e^{-||b\|_1 t}\,\left[U(t)\right]^2\,,\,\,\,U(0) =1$$
so that solving this separable differential inequality gives
$\,U(t)\le e^{\gamma_2t},\,$ which is equivalent to the desired estimate.
\end{proof}

\paragraph{End of proof for Theorem \ref{theoremC}.} Apply Lemma \ref{lemmaMa}.

\subsection{Stability and uniqueness}
For stability of solutions and uniqueness, we prove the following which differs from
Theorem \ref{theorem2} in that results hold in the space $\,\mathcal{S}^\alpha(\R^3\times[0,\infty))\,$
in place of $\,\mathcal{S}^\alpha_p(\R^3\times[0,\infty))\,$ in the cutoff case.

\medskip

\begin{theorem}\label{theoremCS} Let $\,b\in L^1(\s^2)\,$ and $\,0<\alpha\le p.\,$
If $\,\phi, \psi\,$ are solutions to the Cauchy problem (\ref{1.8})
in the space $\,\mathcal{S}^\alpha(\R^3\times[0,\infty))\,$
corresponding to the initial data $\,\phi_0, \psi_0\in \mathcal{K}^\alpha,\,$ respectively, then for all $\,t\ge 0\,$
\begin{equation*}
\sup_{\xi\in\R^3} \left[e^{\delta_p |\xi|^p t}\,\frac{\,|\phi(\xi, t) - \psi(\xi, t)|\,}{|\xi|^\alpha}\right]\le e^{\lambda_\alpha t}
\sup_{\xi\in\R^3}\frac{\,|\phi_0(\xi) - \psi_0(\xi)|\,}{|\xi|^\alpha}.
\end{equation*}
In particular, for any initial datum $\,\phi_0\in\mathcal{K}^\alpha,\,$
the Cauchy problem (\ref{1.8}) has at most one solution in the space $\,\mathcal{S}^\alpha(\R^3\times[0,\infty)).\,$
\end{theorem}

\begin{proof}
Let us put
\begin{align*}
H(\xi, t)  = e^{(\gamma_2 + \delta_p|\xi|^p)t}\left[\frac{\phi(\xi, t) - \psi(\xi, t)}{|\xi|^\alpha}\right]
\end{align*}
for $\,\xi\ne 0\,$ and $\,H(0, t) = 0.\,$ Then $H$ satisfies
$$H(\xi, t) = H(\xi, 0) + \int_0^t  e^{(\gamma_2 + \delta_p|\xi|^p)\tau}\,\left[\frac{\mathcal{G}(\phi)(\xi, \tau)-\mathcal{G}(\psi)(\xi, \tau)}{|\xi|^\alpha}\right]\,d\tau\,.$$
By exploiting the maximum growth estimate of Lemma \ref{lemmaMa}, we deduce
\begin{align*}
& e^{\delta_p|\xi|^p\tau }\bigl|\mathcal{G}(\phi)(\xi, \tau)- \mathcal{G}(\psi)(\xi, \tau)\bigr| \nonumber\\
&\quad\le \int_{\s^2}
b\biggl(\frac{\xi\cdot\sigma}{|\xi|}\biggr) \biggl\{\left[ e^{\delta_p |\xi^+|^p\tau}|\phi(\xi^+, \tau)|\right]
\left[ e^{\delta_p |\xi^-|^p\tau}|\phi(\xi^-, \tau)- \psi(\xi^-, \tau)|\right]\nonumber\\
&\qquad\qquad\qquad\,\, + \,\left[ e^{\delta_p |\xi^-|^p\tau}|\psi(\xi^-, \tau)|\right]
\left[ e^{\delta_p |\xi^+|^p\tau}|\phi(\xi^+, \tau)- \psi(\xi^+, \tau)|\right]\biggr\}\,d\sigma\nonumber\\
&\quad\le e^{-\gamma_2\tau}\,\left\| H(\cdot, \tau)\right\|_\infty\,\int_{\s^2}b\biggl(\frac{\xi\cdot\sigma}{|\xi|}\biggr)
\bigl(\left|\xi^+\right|^\alpha + \left|\xi^-\right|^\alpha\bigr)\,d\sigma.
\end{align*}
Consequently, we find
\begin{equation}\label{G3}
\left\|H(\cdot, t)\right\|_\infty \le \left\|H(\cdot, 0)\right\|_\infty  + \gamma_\alpha\int_0^t \left\|H(\cdot, \tau)\right\|_\infty\,d\tau
\end{equation}
and an easy argument of Gronwall-type yields
$$\left\|H(\cdot, t)\right\|_\infty\,\le\, e^{\gamma_\alpha t}\,\left\|H(\cdot, 0)\right\|_\infty,$$
which is equivalent to the stated estimate.
\end{proof}

\section{Results in the non-cutoff case: \\Proofs of the main theorems}
We are now about to prove our main theorems of existence, uniqueness and stability in the non-cutoff case.
As it is standard in the theory of the Boltzmann equation (see \cite{A} for example), our
proofs are based on an approximation scheme applied to a sequence of solutions obtained in Theorem \ref{theoremC}
for the kernels truncated in certain way.

\subsection{Preliminary time-growth estimate}
Let us remind that the integral equation (\ref{1.9}) is defined as
\begin{equation*}
\phi(\xi, t) =
e^{-\delta_p|\xi|^p t}\phi_0(\xi) + \int_0^t  e^{-\delta_p|\xi|^p(t-\tau)}\mathcal{B}(\phi)(\xi, \tau)\,d\tau.
\end{equation*}
In view of Lemma \ref{lemmaB} and Lemma \ref{lemmaY}, it makes sense pointwise in both space and time variables
for any $\,\phi\in C([0, \infty); \mathcal{K}^\alpha)\,$
satisfying $\,\phi(\xi, \cdot)\in C([0, \infty))\,$ for each fixed $\,\xi\in\R^3\,$ if $b$ satisfies
the condition $\,\mu_\alpha<+\infty.\,$

In the limiting process of our proofs, it will be essential to know the time-growth behavior
of $\,\|\phi(t)-1\|_\alpha\,$ for each solution $\phi$ obtained in Theorem \ref{theoremC}.
To this purpose, we exploit Lemma \ref{lemmaB} of Morimoto to derive

\medskip

\begin{lemma}\label{lemmaP1}
For $\,0<\alpha\le p,$ assume that $\,\mu_\alpha<+\infty\,$ and $\,\phi_0\in\mathcal{K}^\alpha.\,$
If $\phi$ is a solution to (\ref{1.9}) in the space $\,\mathcal{S}^\alpha(\R^3\times [0, \infty)),\,$
then
\begin{align*}
\|\phi(t)-1\|_\alpha\,\le \,e^{5\mu_\alpha t}\bigl[\|\phi_0 -1\|_\alpha + (\delta_p t)^{\alpha/p}\bigr]\qquad(t\ge 0).
\end{align*}
\end{lemma}

\begin{proof}
Writing
\begin{align*}
\frac{\phi(\xi, t) -1}{|\xi|^\alpha} = \frac{e^{-\delta_p|\xi|^p t}\phi_0(\xi) -1}{|\xi|^\alpha}
+ \int_0^t e^{-\delta_p|\xi|^p(t-\tau)}\frac{\mathcal{B}(\phi)(\xi, \tau)}{|\xi|^\alpha}\,d\tau
\end{align*}
and making use of the estimate for $I(\xi, t)$ in the proof of Lemma \ref{lemmaN} and (ii) of Lemma \ref{lemmaB},
it is straightforward to obtain
\begin{equation}
\|\phi(t) -1\|_\alpha \le  \|\phi_0 -1\|_\alpha + (\delta_p t)^{\alpha/p} +5\mu_\alpha\int_0^t\|\phi(\tau) -1\|_\alpha\, d\tau\,.
\end{equation}
A Gronwall-type argument yields the desired estimate in an easy way.
\end{proof}

As an application, we prove the following time-continuity property which will be used
in establishing equicontinuity in time below.

\medskip

\begin{lemma}\label{lemmaP2} Let $\,T>0\,$ and $\,\xi\in\R^3.\,$ Under the same settings
described as in Lemma \ref{lemmaP1}, the solution $\phi$ satisfies
\begin{equation}\label{P2}
 \bigl|\phi(\xi, t) - \phi(\xi, s)\bigr|\, \le\, |t-s|^{\alpha/p}\,|\xi|^\alpha\,\bigl[\delta_p^{\alpha/p} + C_T(\phi_0)\,T^{1-\alpha/p}\bigr]
 \end{equation}
for all $\,s, t\in [0, T],\,$ where $C_T(\phi_0)$ denotes the constant
\begin{equation}\label{P3}
C_T(\phi_0) = 10\mu_\alpha\, e^{5\mu_\alpha T}\left[\|\phi_0 -1\|_\alpha + \left(\delta_p T\right)^{\alpha/p}\right].
\end{equation}
\end{lemma}

\begin{proof}
Assuming $\,s<t,\,$ we write
\begin{align*}
\phi(\xi, t) - \phi(\xi, s) &= \left[e^{-\delta_p|\xi|^p t} - e^{-\delta_p|\xi|^p s}\right]\phi_0(\xi)\\
& + \int_s^t  e^{-\delta_p|\xi|^p(t-\tau)}\mathcal{B}(\phi)(\xi, \tau)\,d\tau\\
&+ \int_0^s  \left[e^{-\delta_p|\xi|^p(t-\tau)}-  e^{-\delta_p|\xi|^p(s-\tau)}\right]\mathcal{B}(\phi)(\xi, \tau)\,d\tau\\
&= {\rm (I) + (II) + (III).}
\end{align*}
It is trivial to see
\begin{equation}\label{P3}
|{\rm(I)}|\,\le\, 1 - e^{-\delta_p|\xi|^p(t-s)}\,\le\, \left(\delta_p |t-s|\right)^{\alpha/p}\,|\xi|^\alpha\,.
\end{equation}
By Lemma \ref{lemmaB} and Lemma \ref{lemmaP1}, we have
$$\left|\mathcal{B}(\phi)(\xi, \tau)\right|\,\le\, 5\mu_\alpha\,\|\phi(\tau)-1\|_\alpha |\xi|^\alpha\,\le\,\frac 12\, C_T(\phi_0)\,|\xi|^\alpha$$
so that we find
\begin{equation}\label{P4}
|{\rm(II)}|\,\le\,\frac 12 \,|t-s|\,C_T(\phi_0)\,|\xi|^\alpha.
\end{equation}
By using the evident estimate
$$\left|e^{-\delta_p|\xi|^p(t-\tau)}-  e^{-\delta_p|\xi|^p(s-\tau)}\right|\,\le\, \delta_p\,|t-s|\,|\xi|^p\,e^{-\delta_p|\xi|^p (s-\tau)},$$
we have
\begin{align}\label{P5}
|{\rm(III)}|\,&\le\,\frac 12\, \delta_p\,|t-s|\,C_T(\phi_0)\,|\xi|^\alpha\,\left(|\xi|^p \int_0^s  e^{-\delta_p|\xi|^p(s-\tau)}\,d\tau\right)\nonumber\\
&\le \,\frac 12\, |t-s|\,C_T(\phi_0)\,|\xi|^\alpha.\,
\end{align}
Adding (\ref{P3})-(\ref{P5}) and simplifying, we obtain the stated bound.
\end{proof}

\subsection{Global existence: Proof of Theorem \ref{theorem1}}
We consider a monotone sequence $(b_n)$ of kernels obtained from $b$ by
cutting off the singularity at $\,\theta=0\,$ in the manner
\begin{equation}\label{L1}
b_n(\cos\theta) = b(\cos\theta)\,\chi_{[1/n,\,\pi/2]}(\theta),\quad n=1,2,\cdots.\footnote{Here $\chi_I$, with an interval $I$, denotes
the indicator function on $I$.}
\end{equation}
Since $b$ is assumed to be at least bounded away from $\,\theta=0,\,$\footnote{Otherwise, we may consider
$\,\widetilde{b}_n = b_n\wedge n\,$ in place of $b_n$.} it is clear that
each $b_n$ is integrable on the unit sphere, $\,b_n\le b\,$ and $\,b_n\to b\,$ monotonically.

For each $n$, let us consider the corresponding equation
\begin{align}\label{P1-1}
&\quad\phi(\xi, t) = e^{-\delta_p|\xi|^p t}\phi_0(\xi) + \int_0^t  e^{-\delta_p|\xi|^p(t-\tau)}\mathcal{B}_n(\phi)(\xi, \tau)\,d\tau,\nonumber\\
&\mathcal{B}_n(\phi)(\xi, \tau) = \int_{\s^2} b_n\biggl(\frac{\xi\cdot\sigma}{|\xi|}\biggr)\left[\phi(\xi^+, \tau)\phi(\xi^-, \tau)-\phi(\xi, \tau)\right]\,d\sigma.
\end{align}
By Theorem \ref{theoremC} and Theorem \ref{theoremCS}, there exists a unique solution $\phi_n$ to (\ref{P1-1}) in the space
$\,\mathcal{S}^\alpha(\R^3\times [0, \infty)\,$ satisfying
\begin{equation}\label{L2}
\sup_{\xi\in\R^3}\,e^{\delta_p|\xi|^p t}\bigl|\phi_n(\xi, t)\bigr|\,\le\, 1\quad\text{for all}\quad t\ge 0.
\end{equation}
In addition, we notice
\begin{align*}
\mu_{n, \,\alpha} = 2\pi\int_0^{\pi/2} b_n(\cos\theta)\sin\theta\sin^\alpha\left(\frac\theta 2\right)\,d\theta\le \mu_\alpha
\end{align*}
so that an application of Lemma \ref{lemmaP1} yields
\begin{equation}\label{L3}
\|\phi_n(t)-1\|_\alpha\le e^{5\mu_\alpha t}\bigl[\|\phi_0 -1\|_\alpha + (\delta_p t)^{\alpha/p}\bigr].
\end{equation}

We observe the following uniform behaviors of $(\phi_n)$:
\begin{itemize}
\item[(i)] (Uniform boundedness) $\,|\phi_n(\xi, t)|\le 1\,$ for all $\,(\xi, t)\in\R^3\times [0, \infty).$
\item[(ii)] (Equicontinuity in space) By (i) of Lemma \ref{lemmaM} and (\ref{L3}),
\begin{align*}
\left|\phi_n(\xi, t) - \phi_n(\eta, t)\right|^2\,&\le\, 2\|\phi_n(t) -1\|_\alpha\,|\xi-\eta|^\alpha\\
&\le 2\,e^{5\mu_\alpha t}\bigl[\|\phi_0 -1\|_\alpha + (\delta_p t)^{\alpha/p}\bigr]\,|\xi-\eta|^\alpha
\end{align*}
uniformly for all $\,\xi, \eta\in\R^3\,$ for each fixed $\,t\ge 0.$
\item[(iii)] (Equicontinuity in time) For each fixed $\,\xi\in\R^3\,$ and $\,T>0,\,$ since $\,\mu_{n,\,\alpha}\le\mu_\alpha,\,$
Lemma \ref{lemmaP2} gives
$$\bigl|\phi_n(\xi, t) - \phi_n(\xi, s)\bigr|\, \le\,|t-s|^{\alpha/p}\,|\xi|^\alpha\,\bigl[\delta_p^{\alpha/p} + C_T(\phi_0)\,T^{1-\alpha/p}\bigr]$$
uniformly for all $\,s, t\in [0, T]\,$ with $C_T(\phi_0)$ defined as in (\ref{P3}).
\end{itemize}

By the Ascoli-Arzel\'{a} theorem and the Cantor diagonal process, owing to (i), (ii),  there exists a subsequence
$\left(\phi_{n_j}\right)$ which converges uniformly on every compact subset of $\,\R^3\times [\mathbb{Q}\cap [0,\infty)].$
Owing to (iii), we may extract a subsequence, still denoted by $\left(\phi_{n_j}\right)$,
which converges uniformly on every compact subset of $\,\R^3\times [0,\infty).\,$ Let
$$\phi(\xi, t) = \lim_{j\to\infty}\,\phi_{n_j}(\xi, t)\quad\text{for}\quad (\xi, t) \in \R^3\times [0,\infty).$$

Let us observe the following properties about the limit $\phi$.
\begin{itemize}
\item[(1)] That $\,\phi(\cdot, t)\in\mathcal{K}\,$ for each $\,t\ge 0\,$ is a direct consequence of L\'evy's continuity theorem due to uniform convergence.
\item[(2)] From (\ref{L3}), we deduce
$$\left|\phi_{n_j}(\xi, t)-1\right|\le |\xi|^\alpha\,e^{5\mu_\alpha t}\bigl[\|\phi_0 -1\|_\alpha + (\delta_p t)^{\alpha/p}\bigr]$$
uniformly in $j$ for all $\,\xi\in\R^3.\,$ Passing to the limit, we have
$$\left|\phi(\xi, t)-1\right|\le |\xi|^\alpha\,e^{5\mu_\alpha t}\bigl[\|\phi_0 -1\|_\alpha + (\delta_p t)^{\alpha/p}\bigr],$$
which implies $\,\phi(\cdot, t)\in\mathcal{K}^\alpha\,$ for each $\,t\ge 0.\,$
\item[(3)] For an arbitrary $\,T>0,\,$ we deduce from (\ref{P2})
$$\bigl|\phi_{n_j}(\xi, t) - \phi_{n_j}(\xi, s)\bigr|\, \le\,|t-s|^{\alpha/p}\,|\xi|^\alpha\,\bigl[\delta_p^{\alpha/p} + C_T(\phi_0)\,T^{1-\alpha/p}\bigr]$$
uniformly for all $\,s, t\in [0, T]\,$ so that passing to the limit yields
$$\bigl|\phi(\xi, t) - \phi(\xi, s)\bigr|\, \le\,|t-s|^{\alpha/p}\,|\xi|^\alpha\,\bigl[\delta_p^{\alpha/p} + C_T(\phi_0)\,T^{1-\alpha/p}\bigr],$$
which implies
$$\bigl\|\phi(t) - \phi(s)\bigr\|_\alpha\, \le\,|t-s|^{\alpha/p}\,\bigl[\delta_p^{\alpha/p} + C_T(\phi_0)\,T^{1-\alpha/p}\bigr]$$
for all $\,s, t\in [0, T].\,$ Reasoning as in the proof of existence for Theorem \ref{theoremC}, we conclude
the map $\,t\mapsto \bigl\|\phi(t) - 1\bigr\|_\alpha\,$ is continuous on $[0, T]$.
\item[(4)] Clearly $\,\phi(\xi, \cdot)\in C([0, \infty))\,$ for each $\,\xi\in\R^3.\,$
\item[(5)] We deduce from (\ref{L2})
 $\,\bigl|\phi_{n_j}(\xi, t)\bigr|\,\le\, e^{-\delta_p|\xi|^p t}\,$
 for all $\,\xi\in\R^3\,$ uniformly so that we conclude by passing to the limit
\begin{equation}
\sup_{\xi\in\R^3}\,e^{\delta_p|\xi|^p t}\bigl|\phi(\xi, t)\bigr|\,\le\, 1\quad\text{for all}\quad t\ge 0.
\end{equation}
\end{itemize}

By properties (1)-(3), $\,\phi\in C([0,\infty); \mathcal{K}^\alpha).\,$
Together with property (4), Lemma \ref{lemmaB} and Lemma \ref{lemmaY} show that
$\mathcal{B}(\phi)$ is well defined and continuous in time. We now prove that $\phi$
verifies the integral equation (\ref{1.9}) and the Cauchy problem (\ref{1.8})
in the classical sense.

Obviously, $\,\phi(0, t) =1\,$ for all $\,t\ge 0.\,$
For fixed $\,\xi\ne 0, \,t>0,\,$ we use the parametrization of the unit sphere described as in section 3 to write
\begin{equation}\label{P1-2}
\int_0^t  e^{-\delta_p|\xi|^p(t-\tau)}\mathcal{B}_{n_j}(\phi_{n_j})(\xi, \tau)\,d\tau = \int_0^t\int_0^{\pi/2}\Phi_{n_j}(\xi, \theta, \tau)\,d\theta d\tau
\end{equation}
in which the integrands are expressed as
\begin{align*}
\Phi_{n_j}(\xi, \theta, \tau) &= e^{-\delta_p|\xi|^p(t-\tau)} b_{n_j}(\cos\theta)\sin\theta\,A(\phi_{n_j})(\xi, \theta, \tau)\quad\text{where}\\
A(\phi_{n_j})(\xi, \theta, \tau) &= \int_{\s^1(\xi)} \bigl[\phi_{n_j}(\xi^+,\tau)\phi_{n_j}(\xi^-,\tau) - \phi_{n_j}(\xi,\tau)\bigr]\,d\omega.
\end{align*}
By the first part of Lemma \ref{lemmaB} and (\ref{L3}), the integrands are dominated by
\begin{align*}
\left|\Phi_{n_j}(\xi, \theta, \tau)\right|
&\le 10\pi  |\xi|^\alpha \,e^{5\mu_\alpha \tau}\bigl[\|\phi_0 -1\|_\alpha + (\delta_p \tau)^{\alpha/p}\bigr]\,b(\cos\theta)\sin\theta\sin^\alpha\left(\frac\theta 2\right)\\
&\equiv \Phi(\xi, \theta, \tau).
\end{align*}
Since it is clear from the definition of $\mu_\alpha$ that
\begin{align*}
\int_0^t\int_0^{\pi/2}\Phi(\xi, \theta, \tau) d\theta d\tau\le 5\mu_\alpha |\xi|^\alpha
\int_0^t e^{5\mu_\alpha \tau}\bigl[\|\phi_0 -1\|_\alpha + (\delta_p \tau)^{\alpha/p}\bigr] d\tau <+\infty,
\end{align*}
we may appeal to Lebesgue's dominated convergence theorem to obtain
\begin{align*}
\lim_{j\to\infty} &\int_0^t  e^{-\delta_p|\xi|^p(t-\tau)}\mathcal{B}_{n_j}(\phi_{n_j})(\xi, \tau)\,d\tau \\
&= \int_0^t\int_0^{\pi/2} e^{-\delta_p|\xi|^p(t-\tau)} b(\cos\theta)\sin\theta\,A(\phi)(\xi, \theta, \tau)\,d\theta d\tau\\
&= \int_0^t  e^{-\delta_p|\xi|^p(t-\tau)}\mathcal{B}(\phi)(\xi, \tau)\,d\tau.
\end{align*}
Therefore
\begin{align*}
\phi(\xi, t) &=  e^{-\delta_p|\xi|^p t}\phi_0(\xi) + \lim_{j\to\infty} \int_0^t  e^{-\delta_p|\xi|^p(t-\tau)}\mathcal{B}_{n_j}(\phi_{n_j})(\xi, \tau)\,d\tau \\
&= e^{-\delta_p|\xi|^p t}\phi_0(\xi) + \int_0^t  e^{-\delta_p|\xi|^p(t-\tau)}\mathcal{B}(\phi)(\xi, \tau)\,d\tau,
\end{align*}
that is, $\phi$ satisfies (\ref{1.9}). Differentiating under the integral sign, we conclude that $\phi$ satisfies the Cauchy problem (\ref{1.8})
and $\,\partial_t\phi(\xi, \cdot)\in C((0, \infty)).\,$ Thus $\,\phi\in\mathcal{S}^\alpha(\R^3\times[0, \infty))\,$ and
the proof of Theorem 1.1 is now complete.

\subsection{Stability: Proof of Theorem \ref{theorem2}}
We shall prove the stability inequality for $\,t\in[0, T]\,$ with an arbitrarily fixed $\,T>0.\,$
With the same sequence of kernels $(b_n)$ defined as in (\ref{L1}), we consider
the associated operators $\,(\mathcal{G}_n),\,(\mathcal{R}_n)\,$ defined by
\begin{align*}
\mathcal{G}_n(\phi)(\xi) &=  \int_{\s^2} b_n\biggl(\frac{\xi\cdot\sigma}{|\xi|}\biggr) \phi(\xi^+)\phi(\xi^-)\,d\sigma,\\
\mathcal{R}_n(\phi)(\xi) &=  \int_{\s^2}
b_n^r\biggl(\frac{\xi\cdot\sigma}{|\xi|}\biggr) \bigl[\phi(\xi^+)\phi(\xi^-) -\phi(\xi)\bigr]\,d\sigma
\end{align*}
where we put $\, b_n^r = b - b_n.\,$ Since $\phi, \psi$ are solutions to (\ref{1.8}), we have
\begin{align*}
\partial_t(\phi -\psi) + &\left(\|b_n\|_1 + \delta_p|\xi|^p\right)(\phi-\psi)\nonumber\\
&= \left[\mathcal{G}_n(\phi) - \mathcal{G}_n(\psi)\right] + \left[\mathcal{R}_n(\phi) - \mathcal{R}_n(\psi)\right]
\end{align*}
for which we denote
$$\|b_n\|_1 = 2\pi\int_{1/n}^{\pi/2} b(\cos\theta)\sin\theta\,d\theta <+\infty.$$
Upon setting
\begin{equation*}
U(\xi, t) = e^{\delta_p|\xi|^p t}\left[\frac{\phi(\xi, t) - \psi(\xi, t)}{|\xi|^\alpha}\right]
\end{equation*}
for $\,\xi\ne 0\,$ and $\,U(0, t) = 0,\,$
the above identity implies
\begin{align}\label{G4}
&\bigl|\partial_t \left[e^{\|b_n\|_1 t}\,U(\xi, t)\right]\bigr|\nonumber\\
&\qquad\le \,e^{(\|b_n\|_1 + \delta_p|\xi|^p) t}\biggl\{\left| \frac{\mathcal{G}_n(\phi) - \mathcal{G}_n(\psi)}{|\xi|^\alpha}\right| +
\left| \frac{\mathcal{R}_n(\phi) - \mathcal{R}_n(\psi)}{|\xi|^\alpha}\right|\biggr\}.
\end{align}

Let $\,\rho>0\,$ be arbitrary. Put $\,U_\rho(t) = \sup_{|\xi|\le \rho}|U(\xi, t)|\,$ and
\begin{align*}
\gamma_{n, \,\alpha} &= \int_{\s^2} b_n\left(\frac{\xi\cdot\sigma}{|\xi|}\right)
\left(\frac{|\xi^+|^\alpha + |\xi^-|^\alpha}{|\xi|^\alpha}\right)d\sigma\\
&=2\pi\int_{1/n}^{\pi/2} b(\cos\theta)\sin\theta\left[\cos^\alpha\left(\frac\theta 2\right) +
 \sin^\alpha\left(\frac\theta 2\right)\right] d\theta.
\end{align*}
Proceeding in the same manner as in the proof of Theorem \ref{theoremCS}, we find
\begin{equation}\label{G5}
e^{(\|b_n\|_1 + \delta_p|\xi|^p)t}\left| \frac{\mathcal{G}_n(\phi) - \mathcal{G}_n(\psi)}{|\xi|^\alpha}\right|
\le \gamma_{n, \,\alpha}\, e^{\|b_n\|_1 t} U_\rho(t)
\end{equation}
for all $\,|\xi|\le\rho.$ On the other hand, Lemma \ref{lemmaB} gives
\begin{align*}
\bigl|\mathcal{R}_n(\phi)(\xi, t)\bigr|
= 10\pi\|\phi(t) -1\|_\alpha |\xi|^\alpha\,\int_0^{1/n} b(\cos\theta)\sin\theta\sin^\alpha\left(\frac\theta 2\right)\,d\theta.
\end{align*}
By considering similar estimate for $\mathcal{R}_n(\psi)$, hence, we note
\begin{equation}\label{G6}
e^{(\|b_n\|_1 + \delta_p|\xi|^p)t}\left| \frac{\mathcal{R}_n(\phi) - \mathcal{R}_n(\psi)}{|\xi|^\alpha}\right|
\le K_{n}\, e^{\|b_n\|_1 t}
\end{equation}
for all $\,|\xi|\le \rho,\,t\in[0, T],\,$ where we put
\begin{align*}
K_n &= M_\rho(T)\,\int_0^{1/n} b(\cos\theta)\sin\theta\sin^\alpha\left(\frac\theta 2\right)\,d\theta\,,\\
M_\rho(T) &= 10\pi\,e^{\delta_p\,\rho^p\,T}\,\max_{t\in[0, T]}\left(\|\phi(t)-1\|_\alpha + \|\psi(t)-1\|_\alpha\right)\,.
\end{align*}
In view of the growth estimate of Lemma \ref{lemmaP1}, $\,M_\rho(T)<+\infty\,$ and so an application of Lebesgue's dominated convergence theorem
shows $\,K_n\to 0\,$ as $\,n\to\infty\,$ under the assumption $\,\mu_\alpha<+\infty.\,$

Now the estimates (\ref{G5}) and (\ref{G6}) imply
\begin{equation}\label{G7}
\bigl|\partial_t \left[e^{\|b_n\|_1 t}\,U(\xi, t)\right]\bigr|\,\le\, \gamma_{n, \,\alpha}\, e^{\|b_n\|_1 t} U_\rho(t)
+ K_{n}\, e^{\|b_n\|_1 t}
\end{equation}
for all $\,|\xi|\le \rho,\,t\in[0, T].\,$ If we put $\,W(t) = e^{\|b_n\|_1 t} U_\rho(t),\,$ then it yields
$$W(t)\le W(0) + \gamma_{n, \,\alpha}\,\int_0^t W(s) ds + \frac{K_n}{\|b_n\|_1}\left(e^{\|b_n\|_1 t} -1\right)\,.$$
A standard Gronwall-type argument gives
$$
W(t)\le e^{\gamma_{n,\,\alpha} t}\left\{ W(0) +  \frac{K_n}{\gamma_{n, \,\alpha} -\|b_n\|_1}
\left[1 - e^{-(\gamma_{n,\,\alpha} - \|b_n\|_1) t}\right]\right\},$$
or equivalently,
$$U_\rho(t) \le e^{(\gamma_{n,\,\alpha}-\|b_n\|_1) t}\,U_\rho(0) + \frac{K_n}{\gamma_{n, \,\alpha} -\|b_n\|_1}
\left[e^{(\gamma_{n,\,\alpha} - \|b_n\|_1) t} -1\right].$$
Since
$$\gamma_{n,\,\alpha} - \|b_n\|_1 = 2\pi\int_{1/n}^{\pi/2} b(\cos\theta)\sin\theta\left[\cos^\alpha\left(\frac\theta 2\right) +
 \sin^\alpha\left(\frac\theta 2\right) -1\right] d\theta,$$
we notice $\,0<\gamma_{n,\,\alpha} - \|b_n\|_1 \to \lambda_\alpha\,$ increasingly as $\,n\to \infty\,.$ Passing to the limit, we conclude
$\,U_\rho(t) \le e^{\lambda_\alpha t}\,U_\rho(0)\,$ for all $\,t\in [0, T].\,$ Letting $\,\rho\to +\infty,\,$
we finally obtain
$$\sup_{\xi\in\R^3}\,|U(\xi, t)|\,\le\, e^{\lambda_\alpha t}\,\sup_{\xi\in\R^3}\,|U(\xi, 0)|,$$
which is equivalent to the desired stability estimate on $[0, T]$.\qed

\section{Non-existence}
Due to the intrinsic properties of the characteristic functions associated with the symmetric
stable L\'evy processes, it is easy to observe the following negative result which shows
that the Cauchy problem (\ref{1.8}) is ill-posed even locally in our solution space when $\,\alpha>p.$

\medskip

\begin{theorem}
For $\,p<\alpha\le 2,\,$ assume that $b$ satisfies $\,\mu_\alpha<+\infty.\,$
Then for any $\,\phi_0\in\mathcal{K}^\alpha\,$ and $\,T>0,\,$ there does not exist
a solution to the Cauchy problem (\ref{1.8}) in the space
$\,\mathcal{S}^\alpha(\R^3\times [0, T]).\,$
\end{theorem}

\begin{proof}
If there were such a solution $\phi$, then it would satisfy
\begin{align*}
1- e^{-\delta_p|\xi|^p t} &= e^{-\delta_p|\xi|^p t}\bigl[\phi_0(\xi)-1\bigr] \\
&\quad + \bigl[1-\phi(\xi, t)\bigr] + \int_0^t  e^{-\delta_p|\xi|^p(t-\tau)}\mathcal{B}(\phi)(\xi, \tau)\,d\tau
\end{align*}
for all $\,(\xi, t)\in\R^3\times [0, T].\,$ It follows from Lemma \ref{lemmaM} that
\begin{align*}
\sup_{\xi\in\R^3}\left|\frac{1- e^{-\delta_p|\xi|^p t}}{|\xi|^\alpha}\right|
&\le \|\phi_0-1\|_\alpha \\
&\quad + \|\phi(t)-1\|_\alpha + 5\mu_\alpha\int_0^t \|\phi(\tau)-1\|_\alpha\,d\tau\\
&<+\infty
\end{align*}
for each fixed $\,0<t\le T,\,$ a contradiction in view of (\ref{NP}) of Lemma \ref{lemmaW}.
\end{proof}

\section{Further remarks}
For $\,0<\alpha\le 2,\,$ let $\,\tilde{P}_\alpha = \mathcal{F}^{-1}(\mathcal{K}^\alpha),\,$
the space of probability measures whose Fourier transforms belong to $\mathcal{K}^\alpha$ endowed with the
Fourie-based metric $d_\alpha$. Under this setting, Theorem 1.1 may be interpreted as an existence theorem for
a measure-valued solution $\,f\in C([0, \infty); \tilde{P}_\alpha)\,$ to the Cauchy problem (\ref{1.1})
in a weak sense for any initial datum $\,f_0\in \tilde{P}_\alpha.\,$

It is well-known that the singular kernel $b$ of type (\ref{1.4})
entails certain smoothing effects to a solution of the Boltzmann equation. For instance, if the initial datum
$f_0$ is an $L^1_2$ function having finite entropy, then any weak solution $f(v, t)$ to the Boltzmann equation
becomes smooth in the sense $\,f(\cdot, t)\in H^\infty(\R^3)\,$ for all $\,t>0\,$ (see \cite{AE}), where a
weak solution may be defined as in either \cite{Go}, \cite{Vi2} or a solution to (\ref{1.10}).

A natural question is whether such a smoothing effect would occur even for a measure-valued solution.
For this matter, a remarkable progress has been made
recently by Morimoto \& Yang (\cite{MY}) who proved that a measure-valued solution $\,f\in C([0, \infty); \tilde{P}_\alpha)\,$
to the Cauchy problem for the Boltzmann equation indeed satisfies $\,f(\cdot, t)\in H^\infty(\R^3)\,$
locally in time unless the initial measure $\,f_0\in \tilde{P}_\alpha\,$ is not a Dirac measure (globally in time when $\,f_0\in P_2\,$).
By making use of their result, Cannone \& Karch proved in \cite{CK1} that self-similar solutions constructed by
Bobylev \& Cercignani are smooth in this sense under certain extra assumption on $b$.

In consideration of the additional diffusive term, it may be expected that better smoothing effects would take place
for a measure-valued solution to the Fokker-Planck-Boltzmann equation (\ref{1.1}).

Another noticeable development is the recent work of Morimoto, Wang \& Yang \cite{MWY} on the image space of $P_\alpha$ under the Fourier transform.
As we have mentioned in Remark 1 of Section 2, $\,\mathcal{F}(P_\alpha)\subsetneq\mathcal{K}^\alpha\,$
for $\,0<\alpha<2.\,$ In a simplified version, one of their theorems assert that $\mathcal{F}(P_\alpha)$ is the space of
characteristic functions $\phi$ on $\R^3$ satisfying
\begin{equation*}
\int_{\R^3} \frac{|\phi(\xi)-1|}{|\xi|^{3+\alpha}}\,d\xi <+\infty\qquad(0<\alpha<2, \,\alpha\ne 1).
\end{equation*}
By exploiting this characterization, they also constructed a unique measure-valued solution $\,f\in C([0, \infty); P_\alpha)\,$
to the Cauchy problem for the Boltzmann equation with an initial measure $\,f_0\in P_\alpha.\,$
It will be interesting to explore their characterization and study the Fokker-Planck-Boltzmann equation (\ref{1.1})
from such a point of view.

\bigskip

\noindent
{\bf Acknowledgements.} The author is grateful to the anonymous referee who pointed out the recent developments
described as in the last section. This research was supported by National Research Foundation of Korea Grant
funded by the Korean Government (\# 20130301) and 2013-14 Chung-Ang University sabbatical research grant.

\end{document}